\documentclass{amsart}

\usepackage{graphicx}
\usepackage[all]{xy}
\usepackage{color}
\usepackage{amsmath}
\usepackage{amsfonts}
\usepackage{amsthm}
\usepackage{amssymb}
\usepackage{mathrsfs}
\usepackage{hyperref}

\theoremstyle{theorem}
\newtheorem{thrm}{Theorem}
\newtheorem{lem}[thrm]{Lemma}
\newtheorem{cor}[thrm]{Corollary}
\newtheorem{prop}[thrm]{Proposition}
\newtheorem{conj}[thrm]{Conjecture}

\numberwithin{thrm}{section} 

\def \Dj{\mbox{\raise0.3   ex\hbox{-}\kern-0.4em D}}

\newcommand{\im}{\operatorname{im}}

\newcommand{\Lie}{\operatorname{Lie}}
\newcommand{\Length}{\operatorname{Length }}
\newcommand{\divergence}{\operatorname{div}}
\newcommand{\Bx}{\operatorname{Box}}

\newcommand{\Vect}{\operatorname{Vect}}
\newcommand{\rank}{\operatorname{rank}}

\theoremstyle{remark}
\newtheorem{rem}[thrm]{Remark}
\newtheorem{exm}[thrm]{Example}
\newtheorem{defn}[thrm]{Definition}

\def\R{\mathbb{R}}

\def\Z{\mathbb{Z}}

\def\G{\mathbb{G}}
\def\F{\mathbb{F}}
\def\H{\mathbb{H}}
\def\T{\mathbb{T}}
\def\B{\mathcal{B}}

\begin{document}

\title[Coarse nodal counts on sub-Riemannian manifolds]{Coarse nodal counts on sub-Riemannian manifolds}

\author[Irene Silvestre-Rosell\'{o}]{Irene Silvestre-Rosell\'{o}}
\address{D\'e\-par\-te\-ment de math\'ematiques et de
sta\-tistique, Univer\-sit\'e de Mont\-r\'eal,  CP 6128 succ
Centre-Ville, Mont\-r\'eal,  QC  H3C 3J7, Canada.}
\email{irene.silvestre.rosello@umontreal.ca}

\author[Vuka\v sin Stojisavljevi\'c ]{Vuka\v sin Stojisavljevi\'c}
\address{Mathematical Institute, University of Oxford, Andrew Wiles Building, Radcliffe Observatory Quarter, Woodstock Road, Oxford, OX2 6GG}
\email{vukasin.stojisavljevic@gmail.com}
\maketitle

\begin{abstract}
We study coarse topology of nodal sets of linear combinations of eigenfunctions of sub-Laplacians. More precisely, we prove coarse versions of Courant's and B\'{e}zout's theorems for linear combinations of eigenfunctions of sub-Laplacians on compact nilmanifolds obtained as quotients of stratified groups. We conjecture the extensions of these results to general closed equiregular sub-Riemannian manifolds and outline a programme for proving them. The method we use combines topological persistence and anisotropic Sobolev theory of H\"{o}rmander vector fields, generalizing the ideas which have recently been implemented in the Riemannian case.
\end{abstract}

\tableofcontents

\section{Introduction}


The present article is concerned with coarse topology of the zero sets of linear combinations of eigenfunctions of sub-Laplacians and, more generally, maximally hypoelliptic operators. On the topological side, our perspective is rooted in \textit{topological persistence}. This is an emerging field of mathematics, which was initially developed as a part of data analysis, and has since grown into an independent subject. In our context, topological persistence offers a way to systematically define and study topological invariants of functions, which are robust with respect to $C^0$-perturbations (hence the adjective ``coarse").  On the side of spectral geometry, we take inspiration from a plethora of Courant-type theorems and questions, which date back to the Courant's celebrated nodal domain theorem. These ideas have recently been explored in the setting of sub-Riemannian geometry and our results can be considered a contribution to this research direction.

\subsection{Different flavours of Courant's theorem} Let $(M,g)$ be a closed, $n$-dimensional, Riemannian manifold. The positive Laplace-Beltrami operator $\Delta_g = - \divergence_g \circ \nabla_g$ has a discrete spectrum which tends to infinity, i.e.  its eigenvalues form a sequence, $0=\lambda_1 < \lambda_2 \leq \lambda_3 \leq \ldots \nearrow$, each $\lambda_i$ having a finite multiplicity. All eigenfunctions of $\Delta_g$ are $C^\infty$-smooth and there is a basis of $L^2(M,g)$ consisting only of such eigenfunctions, see \cite{LMP23} for preliminaries on spectral geometry. Given an eigenfunction $f$, it is customary to call the set $\{f=0\}$ the \textit{nodal set} of $f$, while the connected components of $M\setminus \{f=0\} $ are called \textit{nodal domains} of $f.$ We adopt this terminology for arbitrary continuous functions on $M.$ Denote by $m(f)$ the number of nodal domains of $f.$ Courant's nodal domain theorem states that if $\Delta_g f  =\lambda_i f$ then $m(f)\leq i.$ Combining this theorem with Weyl's law yields for each $f$, $\Delta f = \lambda f$ a bound of the form
\begin{equation}\label{Equation_Classical_Courant}
m(f)\leq C (\lambda+1)^\frac{n}{2},
\end{equation}
the constant $C$ depending only on $(M,g).$

The question of extending Courant's theorem to linear combinations of eigenfunctions is known as \textit{Courant–Herrmann conjecture} (another relevant term is the \textit{extended Courant property}). Given $\lambda>0$, let us denote by $\mathcal{F}_\lambda$ the set of all finite linear combinations of eigenfunctions of $\Delta_g$ with eigenvalues at most $\lambda.$ The most direct generalization of Courant's theorem predicts that if $f\in \mathcal{F}_{\lambda_i}$ then $m(f)\leq i.$ This prediction is correct in the 1-dimensional case,  even for more general Sturm-Liouville problems, as has been proven by Sturm in 1836. Another proof of this result, following an idea of Gelfand, as well as a historical overview can be found in \cite{BH20a,BH20b}. The first counterexample in higher dimensions was found by Viro in \cite{Viro79}, following a question of Arnold - \cite{Arnold73}.  A number of subsequent counterexamples were found in \cite{GZ03, BH18,BH21,BH21b,BCH22}. In light of these developments, it is natural to reformulate the conjecture, for example by asking if the inequality (\ref{Equation_Classical_Courant}) holds for $f\in \mathcal{F}_\lambda.$ A striking counterexample to this circle of questions was given by Buhovsky, Logunov and Sodin in \cite{BLS20}. They constructed a Riemannian metric $g_{BLS}$ on the 2-dimensional torus, such that there exists a finite linear combination of eigenfunctions of $\Delta_{g_{BLS}}$ which has infinitely many nodal domains. This demonstrated that there cannot exist any general bound on $m(f)$ for $f\in \mathcal{F}_\lambda.$

In contrast, positive results have recently been obtained for the \textit{coarse reformulation} of the Courant-Herrmann conjecture. To explain this, let us fix a real parameter $\delta>0$, which serves as a measure of coarseness. A nodal domain $U$ of $f$ is called  \textit{$\delta$-deep} if $\max_U |f| >\delta.$ Denote by $m_\delta(f)$ the number of $\delta$-deep nodal domains of $f.$ Coarse Courant's theorem states that for every $f\in \mathcal{F}_\lambda$, it holds that 
\begin{equation}\label{Equation_Coarse_Courant}
m_{\delta}(f)\leq C_{\delta} (\lambda+1)^\frac{n}{2},
\end{equation}
where the constant $C_\delta$ depends only on $(M,g)$ and $\delta.$ This theorem was proven in dimension 2 in \cite{PolSod07} and on a general closed manifold in \cite{BPPPSS22}. In light of the Buhovsky-Logunov-Sodin example, we know that $C_\delta \to +\infty$ as $\delta \to 0$ (the rate of convergence was discussed in \cite{BPPPSS22}). The proof of (\ref{Equation_Coarse_Courant}) given in \cite{BPPPSS22} relies on a connection between Sobolev theory and topological persistence. This idea has direct precursors in \cite{PolSod07,CSEHM, PPS19,Perez22} and more distant precursors in \cite{Kronrod50,Vitushkin55, Yomdin85}. 

\subsection{Courant-type theorems for sub-Laplacians}

So far we discussed Courant-type theorems on Riemannian manifolds. A parallel development is currently ongoing in the sub-Riemannian setting. Let $M$ be a closed manifold and $\xi$ a regular distribution on $M$, i.e. a subbundle of $TM.$ Denote by $C^\infty(M;\xi)$ the space of smooth sections of $\xi$ and by $[\cdot, \cdot]$ the commutator of vector fields.  Let  $\xi^{(0)}=\{0 \}$ be the zero subbundle of $TM$,  $\xi^{(1)}=\xi$ and for every integer $i\geq 2$ and every $p\in M$ let
$$\xi^{(i)}_p = Span_\R (\{ [X_1, [X_2,\ldots , [X_{j-1},X_j]\ldots ]_p~|~ 1\leq j \leq i, ~X_1,\ldots , X_j \in C^\infty(M; \xi))\}.$$
The distribution $\xi$ is called \textit{bracket generating} if there exists an integer $s\geq 1$ such that for all $p\in M$, $\xi^{(s)}_p=T_pM.$ The minimal such $s$ is called the \textit{step} of $\xi.$ A distribution of step $s$ is called \textit{equiregular} if each $\xi^{(i)}$, $1\leq i \leq s$, is a regular distribution on $M$, i.e. if for each $i$, $\dim \xi^{(i)}_p$ does not depend on $p\in M.$ Let us now fix an equiregular distribution $\xi$ on a closed manifold $M.$ The \textit{homogeneous dimension} $Q$ of $(M,\xi)$ is defined as $Q=\sum_{i=1}^s i \cdot (\rank \xi^{(i)}-\rank \xi^{(i-1)}).$ A \textit{sub-Riemannian metric} $g$ on $(M,\xi)$ is a smooth family of scalar products on $\xi.$ We call the triple $(M,\xi,g)$ an \textit{equiregular sub-Riemannian manifold}. Note that in the special case $\xi=TM$, we have that $s=1$, $Q=n$ and we recover the usual notion of a Riemannian manifold. Another widely-studied class of equiregular distributions are contact distributions. In this case $n$ is odd, $\xi$ is a maximally non-integrable hyperplane distribution, $s=2$ and $Q=n+1.$ In Section \ref{Section_Sub-Riemannian}, we briefly discuss basic notions of sub-Riemannian geometry and provide further references. Here, we restrict our attention to the spectral geometry of sub-Riemannian manifolds. To this end, let $(M,\xi,g)$ be a closed equiregular sub-Riemannian manifold and let us fix a Borel measure $\mu$ with a smooth positive density. The \textit{horizontal gradient} $\nabla^\xi f (p)$ of a smooth function $f$ at a point $p\in M$ is defined as the unique vector in $\xi_p$ such that $g(\nabla^\xi f(p), X_p)=df(p)(X_p)$ for every $X_p \in \xi_p.$ The \textit{sub-Laplacian} associated to $(M,\xi,g,\mu)$ is the operator $P = - \divergence_\mu \circ \nabla^\xi.$ The study of sub-Laplacians (and more general H\"{o}rmander's sub-Laplacians, see Section \ref{Section_Sub-Riemannian}), was initiated in \cite{Hormander67}, followed by a number of influential works. In analogy to the case of Laplace-Beltrami operators, each sub-Laplacian $P$ has a discrete spectrum which tends to infinity, i.e. there is an increasing diverging sequence of eigenvalues, each having finite multiplicity. Furthermore, the eigenfunctions of $P$ are $C^\infty$-smooth and there is a basis of $L^2(M,\mu)$ consisting of them. In order to complete the analogy with the Riemannian case, we need to discuss two results - Weyl's law and Courant's nodal domain theorem.

For $\lambda>0$, let $\frak{n}(\lambda)$ denote the number of eigenvalues of $P$ not greater than $\lambda$, counting multiplicities.  The Riemannian Weyl's law, i.e. the case $\xi=TM$, states that 
$$\frak{n}(\lambda)= const \cdot \lambda^{\frac{n}{2}}+o(\lambda^{\frac{n}{2}}),$$
as  $\lambda \to \infty.$ The generalization of this result to equiregular sub-Riemannian manifolds was proven by M\'{e}tivier in \cite{Metivier76}.  It states that 
$$\frak{n}(\lambda)= const \cdot \lambda^{\frac{Q}{2}} + o(\lambda^{\frac{Q}{2}}).$$
An interesting feature of this result is the appearance of the homogeneous dimension $Q$ in the exponent. It reflects a non-obvious connection to the sub-Riemannian geometry of $(M,\xi,g).$ Indeed, $Q$ can be interpreted as Hausdorff dimension of $(M,d_{cc})$, where $d_{cc}$ is the \textit{Carnot-Carath\'{e}odory distance}, defined by minimizing the lengths of curves tangent to $\xi$, see Section \ref{Section_Sub-Riemannian} and references therein for details. Further sub-Riemannian generalizations of Weyl's law have since been proven, see \cite{CdVHT22} and references therein.

In regards to Courant-type theorems, a bound on the number of nodal domains of eigenfunctions of sub-Laplacians was proven in \cite{EL23}, see also \cite{FH25a,FH25b,FH26} for related results. Combined with Weyl's law, the bound from \cite{EL23} implies a sub-Riemannian generalization of (\ref{Equation_Classical_Courant}). Namely,  for $f$ such that $P f= \lambda f$ it holds that 
\begin{equation}\label{Equation_Courant_Sub}
m(f)\leq C (\lambda+1)^\frac{Q}{2},
\end{equation}
where $C$ depends only on $(M,g,\xi)$ and $\mu.$ The fact that the Courant-Herrmann conjecture fails already in the Riemannian case discourages the search for such a result in the sub-Riemannian setting. However, we expect that the coarse reformulation should extend to this setting. More precisely, denoting by $\mathcal{F}_\lambda$ the space of finite linear combinations of eigenfunctions of $P$ with eigenvalues at most $\lambda$, we conjecture, see Conjecture \ref{Conjecture_Courant}, that for all $\delta>0$ and $f\in \mathcal{F}_\lambda$ it holds that
\begin{equation}\label{Equation_Coarse_Sub_Courant}
m_\delta(f)\leq C_{\delta} (\lambda+1)^\frac{Q}{2},
\end{equation}
with $C_{\delta}$ depending on $(M,g,\xi)$, $\mu$ and $\delta.$

As one of our main results, we prove (\ref{Equation_Coarse_Sub_Courant}) on nilmanifolds obtained as quotients of stratified groups by lattices, albeit with a slightly worse exponent $\frac{Q}{2}(1+\varepsilon)$ for arbitrarily small $\varepsilon$, see Theorem \ref{Thm:Courant_Nilmanifolds}. Since stratified groups serve as local models for general equiregular structures, we believe that this result is not only a proof of the special case of a general conjecture, but rather a concrete step towards a complete proof. With that in mind, in Subsection \ref{Intro_Equiregular} we outline a programme for proving (\ref{Equation_Coarse_Sub_Courant}) on any closed equiregular sub-Riemannian manifold.  Let us also mention that since our approach is rooted in topological persistence, it has an advantage of being rather general and applicable to a number of different questions. Indeed, besides the coarse Courant's theorem, we also prove a coarse version of B\'{e}zout's theorem for linear combinations of eigenfunctions of sub-Laplacians, see Theorem \ref{Theorem_Bezout_Main},  an anisotropic Sobolev estimate on barcodes, see Theorem \ref{Thm:Sobolev_Nilmanifolds}, as well as extend all the spectral geometric results to general maximally hypoelliptic operators, see Corollary \ref{Corollary_Maximally_Hypoelliptic}. Lastly,  our bounds apply to some of the higher homological degrees (notably degree $n-1$). Such topological bounds are related to the shape and the relative position of the connected components of the nodal set, which is a topic of independent interest, see \cite{GSKL26, Koirala26}. 

\subsection{Coarse Courant's theorem}
Let $\G=(G,\cdot)$ be a real Lie group and $\frak{g}$ its Lie algebra, which we identify with left-invariant vector fields on $G.$ A \textit{stratification} of $\frak{g}$ is a direct sum decomposition
$$\frak{g}=V_1 \oplus \ldots \oplus V_s,$$
such that $V_s\neq \{0 \}$ and for all $1\leq j\leq s-1$,
$$
V_{j+1}=Span_\R ([V_1,V_j]).
$$
A Lie group together with a stratification of its Lie algebra is called \textit{stratified}.  A \textit{Carnot group} is a stratified Lie group $\G$ together with a scalar product on $V_1.$\footnote{Some authors do not make a distinction between stratified and Carnot groups. Nevertheless, it will be convenient for us to separate the two notions.} Given a Carnot group $\G$, the vector fields from $V_1$ generate an equiregular distribution on $G$ whose step is $s.$ Furthermore, the scalar product on $V_1$ defines an equiregular sub-Riemannian structure on $G$, see Sections \ref{Section_Stratified} and \ref{Section_Sub-Riemannian} for details. A \textit{lattice} $\Gamma$ is a discrete subgroup of $\G$ such that the quotient $\Gamma \backslash \G$ is compact. Such a quotient is called a \textit{nilmanifold}.  It is a closed equiregular sub-Riemannian manifold, whose sub-Riemannian structure is induced from $\G.$ Basic examples of these notions are the real Heisenberg group and the Heisenberg nilmanifold, obtained as a quotient of the real Heisenberg group by the integer Heisenberg group, see Example \ref{Example_Heisenberg}.

Let us fix a Carnot group $\G$, a lattice $\Gamma < \G$ and the corresponding nilmanifold $N=\Gamma \backslash \G.$ Denote by $n$ the topological dimension of $N$ and by $Q$ the homogeneous dimension of $N.$ We equip $\G$ with a bi-invariant Haar measure $\mu$ and $N$ with the induced measure $\nu$, i.e. the unique measure whose density locally satisfies $\rho_\mu=\pi^* \rho_\nu$, $\pi:G\to N$ being the projection. Let $P$ be the sub-Laplacian on $N$ associated to $\nu.$ As before, for $\lambda>0$, denote by $\mathcal{F}_\lambda$ the space of all linear combinations of eigenfunctions of $P$ with eigenvalues at most $\lambda.$

Let $l\geq 1$ be an integer and $F:N\to \R^l$, $F=(f_1,\ldots , f_l)$ a continuous map. Our first goal is to study the topology of the set
$$\{F \neq 0\} = \{f_1 \neq 0 \}\cup \ldots \cup \{f_l \neq 0 \}.$$
To this end, we introduce a quantitative topological invariant of this set. The definition applies to arbitrary continuous maps on arbitrary Hausdorff topological spaces, but our main focus will be on maps $F$ for which $f_1,\ldots , f_l \in \mathcal{F}_\lambda$ for a fixed $\lambda>0.$

Denote by $\check{H}_d(\cdot)$ the degree-$d$ \v{C}ech homology with coefficients in an arbitrary field.  Given a threshold $\delta\geq 0$, the inclusion of subsets $\{ |F|>\delta \} \subset \{ F\neq 0 \}$ induces a linear map
$$ \check{H}_d(\{ |F|>\delta \}) \to \check{H}_d(\{ F\neq 0 \}).$$
We define the \textit{degree-d coarse nodal count of} $F$ as 
\begin{equation}\label{Definition_Coarse_Nodal_Count}
m_{d,\delta}(F)= \rank \left( \check{H}_d(\{ |F|>\delta \}) \to \check{H}_d(\{ F\neq 0 \}) \right) .
\end{equation}
This invariant was extensively studied in \cite{BPPPSS22}, with precursors in \cite{PPS19} and \cite{PolSod07}. Its conceptual origin lies in the field of topological persistence, as we will briefly mention in Subsection \ref{Intro_Barcodes} and further explain in Section \ref{Section_TDA}.

Before we state our first result, let us explain the intuition behind the above definition. Recall that a connected component of $\{ F\neq 0 \}$ is called a \textit{nodal domain} of $F.$\footnote{Traditionally, this terminology is used for functions, i.e. when $l=1.$} For $\delta\geq 0$, a nodal domain $U$ is called \textit{$\delta$-deep} if $\max_U |F| >\delta.$ Focusing on degree-0, we have that
$$m_{0,\delta} (F) = \text{the number of } \delta\text{-deep nodal domains of } F.$$
Indeed, $\dim \check{H}_0(\{ F\neq 0 \})$ equals the number of connected components of $\{ F\neq 0 \}$, while $\dim  \check{H}_0(\{ |F|>\delta \})$ equals the number of connected components of $\{ |F| > \delta \}.$ The expression (\ref{Definition_Coarse_Nodal_Count}) signifies the fact that $m_{0,\delta}(F) $ counts connected components of $\{F\neq 0 \}$ which contain a piece of $\{ |F| > \delta \},$ i.e. $\delta$-deep nodal domains of $F.$ Similarly, $m_{d,\delta}(F)$ can be thought of as a count of $\delta$-deep $d$-dimensional cycles inside of $\{F\neq 0 \}.$

\begin{thrm}\label{Thm:Courant_Nilmanifolds}
Let $k>\frac{Q}{2}$ be an integer and $\varepsilon_k=\frac{Q}{2k-Q}.$ Let $d\in \{0,n-1 \}.$ For every $F=(f_1,\ldots, f_l)$, $f_1,\ldots,f_l\in \mathcal{F}_\lambda$, $\|F\|_{L^2}=1$ and all $\delta>0$ it holds that
$$m_{d,\delta} (F) \leq \frac{C}{\delta^{2\varepsilon_k}} \left( \lambda+1 \right)^{\frac{Q}{2}(1+\varepsilon_k)},$$
where $C$ depends only on $N,\nu, k$ and $l.$ 
\end{thrm}

The connection between Theorem \ref{Thm:Courant_Nilmanifolds} and the classical Courant's nodal domain theorem can be seen by setting $l=1$ and $\delta=0.$ In this case $F:N\to \R$,  $\{ |F|>\delta \} = \{ F\neq 0 \}$ and thus $m_{d,0}(F)=\dim \check{H}_d(\{ F\neq 0 \}).$ In particular, $m_{0,0}(F) =m(F)$ equals the number of nodal domains of $F.$ The name \textit{coarse Courant's theorem} refers to the fact that Theorem \ref{Thm:Courant_Nilmanifolds} only applies to $\delta>0.$

\begin{rem}
Theorem \ref{Thm:Courant_Nilmanifolds} is a corollary of a more general Sobolev estimate given by Theorem \ref{Thm:Sobolev_Nilmanifolds}. As a consequence, Theorem \ref{Thm:Courant_Nilmanifolds} extends to sub-Laplacians associated to arbitrary measures with smooth positive densities, see Remark \ref{Remark_Measures}, and even to general maximally hypoelliptic operators, see Corollary \ref{Corollary_Maximally_Hypoelliptic} and the discussion around it.
\end{rem}

\subsection{Coarse B\'{e}zout's theorem}

Let $\lambda>0$ and let $N=\Gamma \backslash \G, Q,\nu,P, \mathcal{F}_\lambda,l, \check{H}_d(\cdot)$ be as before.  We denote by $b_d(N)$ the degree-$d$ Betti number of $N.$ Given a continuous map $F:N\to \R^l$, $F=(f_1,\ldots , f_l)$, our next goal is to study the topology of the set
$$\{F= 0\} = \{f_1 =0 \}\cap \ldots \cap \{f_l=0 \}.$$
We introduce a quantitative topological invariant of this set, in a similar spirit to the definition of the coarse nodal count. Namely, for $\delta\geq 0$ we define the \textit{degree-$d$ coarse count of zeros} as
\begin{equation}\label{Definition_Coarse_Zero_Count}
z_{d,\delta}(F)= \rank \left( \check{H}_d(\{ F=0 \}) \to \check{H}_d(\{ |F|\leq \delta \}) \right).
\end{equation}
The coarse count of zeros has previously been studied in \cite{BPPPSS22} and \cite{BPPSS23}. As in the case of the coarse nodal count, a conceptual framework in which this invariant appears is given by topological persistence. 

In order to provide intuition for this notion, let us again take $d=0.$ In this case
$$z_{0,\delta}(F)= \text{the number of connected components of } \{|F|\leq \delta \} \text{ which intersect } \{ F=0\}.$$
In a similar spirit, $z_{d,\delta}(F)$ counts $d$-dimensional cycles in $\{|F|\leq \delta \}$ which have a representative in $\{F=0 \}.$

\begin{thrm}\label{Theorem_Bezout_Main}
Let $k>\frac{Q}{2}, l\geq 1$ be integers and $\varepsilon_k=\frac{Q}{2k-Q}.$ Let $d\in \{0,n-1 \}.$ For every $F=(f_1,\ldots, f_l)$, $f_1,\ldots,f_l\in \mathcal{F}_\lambda$, $\|F\|_{L^2}=1$ and all $\delta>0$ it holds that
$$z_{d,\delta} (F) \leq \frac{C}{\delta^{2\varepsilon_k}} \left( \lambda+1 \right)^{\frac{Q}{2}(1+\varepsilon_k)} +b_d(N),$$
where $C$ depends only on $N,\nu,k$ and $l.$
\end{thrm}

To explain the connection between Theorem \ref{Theorem_Bezout_Main} and the classical B\'{e}zout's theorem, let us take $l=n$ and $\delta=0.$ In this case, $\{|F|\leq \delta \} = \{ F=0 \}$ and thus $z_{d,0}(F)=\dim \check{H}_d(\{ F= 0 \}).$ Furthermore, since  
$$\{F= 0\} = \{f_1 =0 \}\cap \ldots \cap \{f_n=0 \},$$
for a generic choice of $f_1,\ldots ,f_n$, the set $\{ F=0 \}$ consists of a finite number of points and $z_{0,0}(F)$ counts these points. In an algebraic setting (for example when $f_1,\ldots , f_n$ are real polynomials on $\R^n$), B\'{e}zout's theorem exactly bounds $z_{0,0}(F).$ As in the case of coarse Courant's theorem, the name \textit{coarse B\'{e}zout's theorem} refers to the fact that $\delta>0$ in Theorem \ref{Theorem_Bezout_Main}. Figure \ref{Figure_Coarse_Zeros} illustrates the coarse count of zeros of a map $F:\R^2 \to \R^2.$ The two curves represent the zero sets of $f_1$ and $f_2$, while the two oval shaded regions represent $\{ |F| \leq \delta \}.$ The zero sets have 6 genuine intersections,  inside 2 shaded regions. Hence $z_{0,0}(F)=6$, while $z_{0,\delta}(F)=2.$

\begin{figure}[ht]
	\begin{center}
		\includegraphics[scale=0.65]{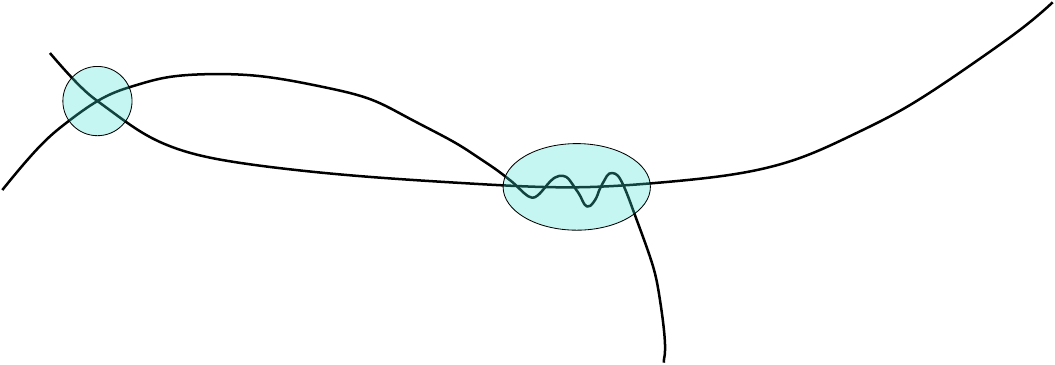}
		\caption{Coarse count of zeros}
		\label{Figure_Coarse_Zeros}
	\end{center}
\end{figure}

\begin{rem}
Analogously to Theorem \ref{Thm:Courant_Nilmanifolds}, Theorem \ref{Theorem_Bezout_Main} extends to sub-Laplacians associated to arbitrary measures with smooth positive densities and to general maximally hypoelliptic operators, see Remark \ref{Remark_Measures}, Corollary \ref{Corollary_Maximally_Hypoelliptic} and the discussion around it.
\end{rem}

\subsection{Anisotropic Sobolev estimates on barcodes}\label{Intro_Barcodes}

Let $\G=(G,\cdot)$ be a stratified group, $\frak{g}$ its stratified Lie algebra and let $X_1,\ldots , X_m$ be a basis of the first layer of stratification of $\frak{g}$ (in particular, $X_i$ bracket-generate the whole $\frak{g}$). As before, let $\Gamma<\G$ be a lattice, $N=\Gamma \backslash \G$ the corresponding nilmanifold and denote by $Q$ the homogeneous dimension of $\G$ and $N.$ The projection $\pi: G\to N$ is a smooth covering and the vector fields $\pi_*(X_1), \ldots , \pi_*(X_m)$ generate a distribution which defines a sub-Riemannian structure on $N.$ Slightly abusing the notation, we denote these vector fields by $X_1,\ldots , X_m$ as well. We fix a bi-invariant Haar measure $\mu$ on $\G$ and the induced measure $\nu$ on $N.$  For $f\in C^\infty (N)$, $p\geq 1$ and an integer $k\geq 1$, let
$$\| X^k f \|_{L^p}= \sum_{1\leq i_1 , \ldots ,   i_k \leq m} \| X_{i_1}\circ \ldots \circ X_{i_k} f \|_{L^p},$$
and
\begin{equation}\label{eq:anisotropic_Sobolev}
 \|f \|_{W^{k,p}_X} = \sum_{i=1}^k \| X^i f \|_{L^p} + \| f\|_{L^p}. 
\end{equation}
Similarly, for a smooth map $F\in C^\infty(N ; \R^l)$, $F=(f_1,\ldots , f_l)$, we denote
\begin{equation}\label{eq:Maps_anisotropic_Sobolev}
 \|F \|_{W^{k,p}_X} = \left( \sum_{i=1}^l \|f_i \|_{W^{k,p}_X}^2 \right)^\frac{1}{2}.
\end{equation}

Given $\delta>0$ and a continuous function $f:N\to \R$, we wish to consider the \textit{persistence barcode of} $f.$ This is a multiset of intervals in $\R$, called \textit{bars}, which describe the topological changes of sublevel sets $\{f \leq t \}$ as $t$ varies in $\R.$ More precisely, for every $s\leq t$, $\rank (\check{H}_d (\{ f\leq s \} ) \to \check{H}_d (\{ f\leq t \} ))$ equals the number of bars in the degree-$d$ barcode of $f$ which contain the interval $[s,t],$ see Section \ref{Section_TDA} for details. Let $\mathcal{N}_{d,\delta}(f)$ be the number of bars of length greater than $\delta$ in the degree-$d$ barcode of $f.$ This invariant was first studied in \cite{CSEHM} and later used in a number of different contexts, the most relevant for us being in \cite{PPS19,BPPPSS22}.

\begin{thrm}\label{Thm:Sobolev_Nilmanifolds}
Let $p\geq 1, l\geq 1$ an integer, $k$ an integer such that $kp>Q$ and $d\in \{0, n-1 \}.$ For every $F\in C^{\infty}(N;\R^l) $ and all $\delta>0$ it holds that
$$\mathcal{N}_{d,\delta} (\pm|F|) \leq C \left( \frac{ \|F \|_{W^{k,p}_X} }{\delta} \right)^{\frac{Qp}{kp-Q}}+b_d(N),$$
where $C$ depends only on $\G,\Gamma,\mu,l,k,p$ and the vector fields $X_1,\ldots,X_m.$ 
\end{thrm}

Theorem \ref{Thm:Sobolev_Nilmanifolds} connects to Theorems \ref{Thm:Courant_Nilmanifolds} and \ref{Theorem_Bezout_Main} via the inequalities
$$m_{d,\delta}(F)\leq \mathcal{N}_{d,\frac{\delta}{2}}(-|F|),$$
$$z_{d,\delta}(F)\leq \mathcal{N}_{d,\frac{\delta}{2}}(|F|),$$
given by Proposition \ref{Proposition_Counts_Comparison}. In general, the discrepancy between $m_{d,\delta}(F), z_{d,\delta}(F)$ and $\mathcal{N}_{d,\frac{\delta}{2}}(\pm|F|)$ can be arbitrarily large, see Example \ref{Example_Discrepancy}. A couple of remarks about Theorem \ref{Thm:Sobolev_Nilmanifolds} are in order.

\begin{rem} The notation $\pm |F|$ in Theorem \ref{Thm:Sobolev_Nilmanifolds} signifies the fact that the upper bound applies to both $\mathcal{N}_{d,\delta} (|F|)$ and $\mathcal{N}_{d,\delta} (-|F|).$ 
\end{rem}

\begin{rem}
The term $+~b_d(N)$ on the right-hand side of the inequality in Theorem \ref{Thm:Sobolev_Nilmanifolds} accounts for infinite bars in the degree-$d$ barcode of $\pm|F|$, see Example \ref{Example_Functions_Pers2}. This term is necessary, which can be seen by taking $F$ to be the zero map.
\end{rem}

\begin{rem}\label{Remark_Measures}
The condition of $\mu$ being a Haar measure in Theorem \ref{Thm:Sobolev_Nilmanifolds} can be loosened to $\mu$ being any measure with a smooth positive density. Indeed, since $N$ is compact, any two Sobolev norms with respect to such measures are equivalent, while $\mathcal{N}_{d,\delta}(\pm |F|)$ and $b_d(N)$ do not depend on the measure.
\end{rem}

\begin{rem}
In the special case $l=1$, $F$ is a continuous function on $N$ and it makes sense to consider $\mathcal{N}_{d,\delta} (F).$ The barcode of $F$ is not related in an obvious way to either of the barcodes of $|F|$ and $-|F|.$ Nevertheless, the proof of Theorem \ref{Thm:Sobolev_Nilmanifolds} can be adapted in a straightforward manner to show the same inequality with $\mathcal{N}_{d,\delta} (F)$ instead of $\mathcal{N}_{d,\delta} (\pm|F|).$ We will not include details of this proof since the bound on $\mathcal{N}_{d,\delta} (F)$ does not relate to our spectral results (coarse Courant's and B\'{e}zout's theorems).
\end{rem}

Coarse Courant's and B\'{e}zout's theorems (Theorems \ref{Thm:Courant_Nilmanifolds} and \ref{Theorem_Bezout_Main}) follow from Theorem \ref{Thm:Sobolev_Nilmanifolds} via hypoelliptic coercivity estimates. Thus, it is natural to extend these results to a more general class of differential operators which satisfy such estimates. To this end, recall that a differential operator of order $r$ of the form
$$\mathcal{D}=\sum_{1\leq i_1 , \ldots ,   i_r \leq m} a_{i_1 , \ldots ,   i_r} \cdot X_{i_1}\circ \ldots \circ X_{i_r},~ a_{i_1 , \ldots ,   i_r} \in C^\infty(N),$$
is called \textit{maximally hypoelliptic} if there exists a constant $C$ such that for every $f\in C^\infty(N)$ it holds that
\begin{equation}\label{Equation_Maximally_Hypoelliptic_Functions}
\| f \|_{W^{r,2}_X} \leq C(\|\mathcal{D} f \|_{L^2} + \| f\|_{L^2}).
\end{equation}
For vector valued maps $F=(f_1,\ldots , f_l)$, we apply $\mathcal{D}$ coordinate-wise, i.e. $\mathcal{D}F=(\mathcal{D}f_1,\ldots,\mathcal{D}f_l).$ It is easy to check that (\ref{Equation_Maximally_Hypoelliptic_Functions}) extends to $F\in C^\infty(N ; \R^l)$ (possibly with a different constant $C$ which also depends on $l$), namely
\begin{equation}\label{Equation_Maximally_Hypoelliptic_Systems}
\| F \|_{W^{r,2}_X} \leq C(\|\mathcal{D} F \|_{L^2} + \| F \|_{L^2}).
\end{equation}
The following is a corollary of Theorem \ref{Thm:Sobolev_Nilmanifolds}.
\begin{cor}\label{Corollary_Maximally_Hypoelliptic} Let $\mathcal{D}$ be a maximally hypoelliptic operator of order $r>\frac{Q}{2}$ and $d\in \{0, n-1 \}.$ For every $F\in C^{\infty}(N;\R^l) $ and all $\delta>0$ it holds that
$$m_{d,\delta} (F) \leq C \left( \frac{ \|\mathcal{D} F \|_{L^2} + \|F \|_{L^2} }{\delta} \right)^{\frac{2Q}{2r-Q}},$$
$$z_{d,\delta} (F) \leq C \left( \frac{ \|\mathcal{D} F \|_{L^2} + \|F \|_{L^2} }{\delta} \right)^{\frac{2Q}{2r-Q}}+b_d(N),$$
where $C$ depends only on $\G,\Gamma,\mu,l$ and $\mathcal{D}.$
\end{cor}
In order to make a connection with Theorems \ref{Thm:Courant_Nilmanifolds} and \ref{Theorem_Bezout_Main}, we recall that the powers of the sub-Laplacian are maximally hypoelliptic, see Theorem \ref{Theorem_Subelliptic_Estimates}. Applying Corollary \ref{Corollary_Maximally_Hypoelliptic} to $\mathcal{D}=P^{\kappa}, \kappa > \frac{Q}{4}$ an integer, yields Theorems \ref{Thm:Courant_Nilmanifolds} and \ref{Theorem_Bezout_Main} for every even $k>\frac{Q}{2}$ ($k=2\kappa$).

\subsection{General equiregular structures}\label{Intro_Equiregular}

Let $(M,\xi, g)$ be a closed $n$-dimensional sub-Riemannian manifold with an equiregular structure $\xi$, $s$ the step of $\xi$, $Q$ its homogeneous dimension and $\mu$ a Borel measure on $M$ with a smooth positive density.  Let $P$ be the sub-Laplacian with respect to $\mu.$ As before, denote by $\mathcal{F}_\lambda$ the space of all linear combinations of eigenfunctions of $P$ with eigenvalues at most $\lambda.$ We conjecture the following version of the coarse Courant's theorem.

\begin{conj}\label{Conjecture_Courant}
Let $k> \frac{Q}{2}$ be an integer and $0\leq d \leq n-1.$ For every $F=(f_1,\ldots, f_l)$, $f_1,\ldots,f_l\in \mathcal{F}_\lambda$, $\|F\|_{L^2}=1$ and all $\delta>0$ it holds that
\begin{equation}\label{Equation_Conj_Courant}
m_{d,\delta} (F) \leq \frac{C}{\delta^{\frac{Q}{k}}} \left( \lambda+1 \right)^{\frac{Q}{2}},
\end{equation}
where $C$ depends only on $(M,\xi, g),\mu, k$ and $l.$
\end{conj}

A couple of remarks are in order. In the case of a Riemannian manifold $(M,g)$, i.e. when $\xi=TM$, Conjecture \ref{Conjecture_Courant} was settled in \cite[Theorem 1.1]{BPPPSS22}, see \cite{PolSod07,PPS19} for precursors. Furthermore, it was shown that in this case, the orders of $\lambda$ and $\delta$ in (\ref{Equation_Conj_Courant}) are almost sharp, see \cite[Theorem 1.10]{BPPPSS22} and \cite[Remark 7.5]{BPPPSS22} for precise statements, as well as \cite[Example 1.4.3]{PPS19}, \cite[Example 1.4.10]{PPS19} and \cite[Proposition 1.4.11]{PPS19}.

On the other hand, in the case of nilmanifolds, Theorem \ref{Thm:Courant_Nilmanifolds} provides strong evidence in support of Conjecture \ref{Conjecture_Courant}. Indeed,  the order of $\lambda$ in Theorem \ref{Thm:Courant_Nilmanifolds}  is $\frac{Q}{2}(1+\varepsilon_k)$ and $\varepsilon_k \to 0$ as $k \to \infty.$ Thus, by picking a sufficiently large $k$, we can make the order of $\lambda$ differ from the conjectured one by an arbitrarily small multiplicative factor. Similarly, the order of $\delta$ is $-2 \varepsilon_k= -\frac{Q}{k-\frac{Q}{2}} $ and $\frac{\frac{Q}{k}}{2\varepsilon_k}\to 1$ as $k \to \infty.$

Let us also note the order of $\lambda$ in the sub-Riemannian Courant estimate (\ref{Equation_Courant_Sub}) agrees with the order of $\lambda$ in the conjectured coarse estimate (\ref{Equation_Conj_Courant}).  As further evidence towards Conjecture \ref{Conjecture_Courant}, we prove the following, non-optimal, bound.

\begin{thrm}\label{Thm:Spectral_RotStein}
Let $k> \frac{Q}{2}$ be an integer and $0\leq d \leq n-1.$ For every $F=(f_1,\ldots, f_l)$, $f_1,\ldots,f_l\in \mathcal{F}_\lambda$, $\|F\|_{L^2}=1$ and all $\delta>0$ it holds that
$$m_{d,\delta} (F) \leq \frac{C}{\delta^\frac{n}{k}} (\lambda+1)^{\frac{sn}{2}}, $$
where $C$ only depends on $(M,\xi, g),\mu, k$ and $l.$
\end{thrm}

Theorem \ref{Thm:Spectral_RotStein} is proven by applying sub-elliptic estimates to the usual Sobolev bound on $m_{d,\delta} (F)$ obtained in \cite{BPPPSS22}. In the Riemannian case, i.e. when $\xi=TM,s=1$, Theorem \ref{Thm:Spectral_RotStein} recovers the optimal bound obtained in \cite{BPPPSS22}. However, in the genuinely sub-Riemannian case, $s>1$ which implies that $sn>Q$ and the order of $\lambda$ in Theorem \ref{Thm:Spectral_RotStein} is always strictly greater than the one in Conjecture \ref{Conjecture_Courant}.  This indicates that proving Conjecture \ref{Conjecture_Courant} requires genuinely sub-Riemannian modifications of the method used in \cite{BPPPSS22} (as does our proof of Theorem \ref{Thm:Courant_Nilmanifolds}).

We will now outline a potential strategy for proving Conjecture \ref{Conjecture_Courant} based on the ideas developed in \cite{BPPPSS22} and the present paper.  To this end, first notice that we obtained Theorem \ref{Thm:Courant_Nilmanifolds} as a corollary of the anisotropic Sobolev barcode estimate given by Theorem \ref{Thm:Sobolev_Nilmanifolds}. Similarly, the Riemannian coarse nodal estimate proven in \cite{BPPPSS22} is a consequence of a Sobolev barcode estimate given by \cite[Theorem 1.12]{BPPPSS22}. We conjecture a common generalization of both of these results.

Let $M$ be a closed manifold, $\xi$ an equiregular distribution on $M$, $Q$ its homogeneous dimension and $\mu$ a Borel measure on $M$ with a smooth positive density.  There exists a finite collection of smooth vector fields $X_1,\ldots , X_m$ on $M$ such that for every $x\in M$, $\xi_x=Span_\R(X_1(x),\ldots, X_m(x)),$ see Section \ref{Section_Sub-Riemannian} for details. We call such vector fields \textit{generating for $\xi.$} The anisotropic Sobolev norm $\| \cdot \|_{W^{k,p}_X}$ with respect to $X_1,\ldots , X_m$ is defined in the same way as in (\ref{eq:Maps_anisotropic_Sobolev}).

\begin{conj}\label{Conjecture_Sobolev}
Let $p\geq 1, l\geq 1$ an integer, $k$ an integer such that $kp>Q$ and $0\leq d \leq n-1.$ For every $F\in C^{\infty}(M;\R^l) $ and all $\delta>0$ it holds that
$$\mathcal{N}_{d,\delta} (\pm|F|) \leq C \left( \frac{\|F \|_{W^{k,p}_X} }{\delta} \right)^{\frac{Q}{k}}+b_d(M),$$
where $C$ depends only on $(M,\xi),\mu,l,k,p$ and the generating vector fields $X_1, \ldots , X_m.$ 
\end{conj}

If proven, Conjecture \ref{Conjecture_Sobolev} would readily imply Conjecture \ref{Conjecture_Courant} via subelliptic estimates. The progress towards Conjecture \ref{Conjecture_Sobolev} is parallel to the progress towards Conjecture \ref{Conjecture_Courant}. It was settled in the Riemannian case in \cite{BPPPSS22}. In the case of sub-Riemannian nilmanifolds, Theorem \ref{Thm:Sobolev_Nilmanifolds} provides an estimate with a worse order of $\frac{ \|F \|_{W^{k,p}_X} }{\delta}.$ Nevertheless, the order of $\frac{ \|F \|_{W^{k,p}_X} }{\delta}$ in Theorem \ref{Thm:Sobolev_Nilmanifolds} is $\frac{Qp}{kp-Q}$, which for sufficiently large $k$ becomes arbitrarily close to the conjectured order of $\frac{Q}{k}$, namely for each $p\geq 1$, $\frac{\frac{Qp}{kp-Q}}{\frac{Q}{k}}\to 1$ as $k\to \infty.$

Similarly to Theorem \ref{Thm:Spectral_RotStein}, a direct application of subelliptic estimates to \cite[Theorem 1.12]{BPPPSS22} yields the following non-optimal result.

\begin{thrm}\label{Thm:Sobolev_RotStein}
Let $p> 1, l\geq 1$ an integer, $k$ an integer such that $kp>n$ and $0\leq d \leq n-1.$ For every $F\in C^{\infty}(M;\R^l) $ and all $\delta>0$ it holds that
$$\mathcal{N}_{d,\delta} (\pm|F|) \leq C \left( \frac{\|F \|_{W^{sk,p}_X} }{\delta} \right)^{\frac{n}{k}}+ b_d(M),$$
where $C$ depends only on $(M,\xi),\mu,l,k,p$ and the generating vector fields $X_1, \ldots , X_m.$ 
\end{thrm}

\begin{rem}
Theorem \ref{Thm:Sobolev_RotStein} also implies non-optimal coarse B\'{e}zout's theorem on equiregular sub-Riemannian manifolds, analogous to Theorem \ref{Theorem_Bezout_Main}. If proven, Conjecture \ref{Conjecture_Sobolev} would imply the optimal coarse B\'{e}zout's bound in this setting.
\end{rem}

Lastly, let us comment on the discrepancy between our results and Conjecture \ref{Conjecture_Sobolev}.  There are three, essentially independent, technical obstacles, which we will now describe. Overcoming each of them would generalize Theorem \ref{Thm:Sobolev_Nilmanifolds} towards Conjecture \ref{Conjecture_Sobolev}, while overcoming all three would prove Conjecture \ref{Conjecture_Sobolev}.

\textbf{Polynomial approximation:} Most of the arguments used in this paper do not rely on the group structure of the Carnot group and, in fact, generalize to closed equiregular sub-Riemannian manifolds. The main exception is a result about local polynomial approximation given by Theorem \ref{Theorem_Polynomial_Approximation}, whose analogue for general equiregular structures, to the best of our knowledge, does not exist. Since Theorem \ref{Theorem_Polynomial_Approximation} is local in nature and Carnot groups serve as local models for the general equiregular structures, we believe it is plausible that a suitable generalization of Theorem \ref{Theorem_Polynomial_Approximation} can be proven. Assuming such a generalization, we expect that the arguments used in this paper would yield an analogue of Theorem \ref{Thm:Sobolev_Nilmanifolds} (and hence also Theorems \ref{Thm:Courant_Nilmanifolds} and \ref{Theorem_Bezout_Main}) in the equiregular setting.

\textbf{Multiscale divisions:} Generalizing Theorem \ref{Thm:Sobolev_Nilmanifolds} to the equiregular setting would still not completely settle Conjecture \ref{Conjecture_Sobolev}, due to the discrepancy in the power of $\frac{ \|f \|_{W^{k,p}_X} }{\delta}$ ($\frac{Qp}{kp-Q}$ as opposed to $\frac{Q}{k}$). In order to bridge this gap in the Riemannian case, in \cite{BPPPSS22} polynomial approximation is applied on multiple scales. The multiscale argument relies on a hierarchy of self-similar subdivisions of cubes in $\R^n$ given by dyadic partitions. We are not aware of a sub-Riemannian analogue of such a hierarchy.

\textbf{Covers with controlled intersections:} The arguments we used in this paper are not directly applicable to the intermediate homological degrees, i.e. $1\leq d \leq n-2$, even in the case of nilmanifolds.  This is due to the fact that they rely on an improved subadditivity theorem for barcodes, which only holds in degree zero, see Theorem \ref{Therem_Subadditivity_Deg_0} and Remark \ref{Remark_Subadditivity_Higher_Deg} (the case $d=n-1$ is covered by the case $d=0$ via Proposition \ref{Proposition_Duality_Barcodes}).  A potential strategy for overcoming this issue would be to use the subadditivity result in all degrees, which has been proven in \cite{BPPPSS22}. In order to do so, one would need to control the geometry of the intersections of the sets in the cover used in the proof of Theorem \ref{Thm:Sobolev_Nilmanifolds}.

\section*{Funding}

During the writing of the article I. S.-R.  was supported by the doctoral scholarship of the Institut des Sciences Math\'{e}matiques. V.S. is a member of the Oxford/Max Planck collaboration and this research was funded in whole or in part by EPSRC international centre to centre collaboration grant EP/Z531224/1. During the writing of the article V. S. was also supported by the CRM-ISM postdoctoral fellowship and Fondation Courtois. V. S. would also like to thank the Isaac Newton Institute for Mathematical Sciences, Cambridge, for support and hospitality during the programme Geometric spectral theory and applications, where work on this paper was undertaken. This work was supported by EPSRC grant EP/Z000580/1.

\section*{Acknowledgments}

We cordially thank Iosif Polterovich for numerous discussions throughout the whole process of writing this article. These discussions helped improve our overall understanding of the subject, as well as the exposition. This project has been inspired by a question of Bernard Helffer. In addition to this, we are very grateful to him for numerous useful comments and especially for sharing his expertise on hypoelliptic operators.  We thank Denis Vinokurov for helpful discussions. This research is part of the first named author Ph.D. thesis at the Universit\'{e} de Montr\'{e}al. 

\section{Topological persistence}\label{Section_TDA}

In this section we present necessary background about topological persistence. This is an active field of research and recently a number of comprehensive introductory texts have appeared, see \cite{Oudot15,CdSGO16,PRSZ20}. Our aim is to make the exposition self-contained by collecting the definitions and results we will use, while referring the reader to Sections 2, 3 and 4 of \cite{BPPPSS22} for complete proofs.  We only present the full proofs of the results which cannot be found in \cite{BPPPSS22}. 

\subsection{Persistence modules and barcodes}

Let $\F$ be a fixed field.

\begin{defn}\label{Definition_Pers_Mod}
A persistence module $(V,\pi)$ is a family of vector spaces $\{V_t \}_{t\in \R}$ together with a family of linear maps $\pi_{s,t}: V_s \to V_t$ defined for all $s\leq t$ such that
\begin{enumerate}
\item $(\forall~ t\in \R)~\pi_{t,t}=id_{V_t};$
\item $(\forall~ r\leq s \leq t)~\pi_{s,t}\circ \pi_{r,s} =\pi_{r,t}.$
\end{enumerate}
\end{defn}

The following is the main example which we will consider throughout the paper.

\begin{exm}\label{Example_Functions_Pers}
Let $X$ be a compact Hausdorff topological space and $f:X\to \R$ a continuous function. Denote by $\check{H}_k(\cdot)$ the degree-$k$ \v{C}ech homology with coefficients in $\F.$ For all $t\in \R$ let $V_k(f)_t = \check{H}_k(\{ f\leq t \}).$ If $s\leq t$ then $\{ f\leq s \} \subset \{f \leq t \}$ and maps $\pi_{s,t} : \check{H}_k(\{f\leq s \} )\to  \check{H}_k(\{f\leq t \} ) $ induced by inclusions render $V_k(f)$ a persistence module.
\end{exm}

\begin{rem}
The choice of  \v{C}ech homology in Example \ref{Example_Functions_Pers} is made for technical reasons, due to the fact that this homology theory has the Mayer-Vietoris sequence for each pair of compact subsets of $X.$ One may define a different persistence module associated to $f$ by considering singular homology of sublevel sets instead. This choice will not make a difference in the end result since, when $X$ is a compact manifold, the bar-counting invariant $\mathcal{N}_\delta$ does not depend on it, see \cite[Proposition 2.12]{BPPPSS22} or \cite{Schmahl25}.
\end{rem}

Let $(V,\pi^V)$, $(W,\pi^W)$ be two persistence modules. A \textit{ morphism of persistence modules} $\phi:(V,\pi^V) \to (W,\pi^W)$ is a family of linear maps $\phi_t:V_t \to W_t$, defined for all $t\in \R$, such that for all $s\leq t$ it holds that $\pi^W_{s,t} \circ \phi_s = \phi_t \circ \pi^V_{s,t}.$ This morphism is an isomorphism if $\phi_t$ is an isomorphism for all $t\in \R.$ The direct sums, direct products, $\ker \phi$, $\im \phi,$ etc. are defined in the same spirit, by considering the relevant notion on vector spaces $V_t, W_t$ for each $t\in \R.$ 

From a more conceptual viewpoint, these definitions can be expressed in the following way. Let $(\R, \leq)$ denote the category whose objects are real numbers and which has a unique morphism $s\to t$ whenever $s\leq t.$ A persistence module $V$ is a functor $V:(\R,\leq) \to \Vect_{\F}$ into the category of $\F$-vector spaces. A morphism of persistence modules is a natural transformation of functors. In other words, the category of persistence modules is a functor category $ \Vect_{\F}^{(\R,\leq)}$ , which is abelian, with the abelian structure induced from the target category $\Vect_{\F}.$

The notion of a persistence module given by Definition \ref{Definition_Pers_Mod} is a bit too general to have a sufficiently rich theory.  To this end, we restrict our attention to a class of \textit{$q$-tame persistence modules}, given by the following definition.

\begin{defn}
A persistence module $(V,\pi)$ is called \textit{$q$-tame} if for every $s<t$ the map $\pi_{s,t}$ has finite rank.
\end{defn}

The assumption of $q$-tameness is essential and most of the theory which we will present exists, in some form, at this level of generality. Nevertheless, in order to simplify the exposition we impose additional assumptions which will be satisfied by all persistence modules considered in this paper.

\begin{defn}
A persistence module $(V,\pi)$ is called \textit{upper semi-continuous} if for every $t\in \R$ the canonical map $V_t\to \lim_{s>t} V_s$ is an isomorphism.  Here, the inverse limit is taken with respect to the structure maps $\pi.$
\end{defn}

Upper semi-continuous, $q$-tame, persistence modules have a canonical decomposition into simple pieces, which we will now describe. Let $I\subset \R$ be an interval. The \textit{interval module} $(\F_I, \pi^{\F_I})$ is a persistence module given by
$$(\F_I)_t =\begin{cases}
      \F, & \text{if}\ t\in I \\
      0, & \text{otherwise} \\
      \end{cases} ~~~~~,~~~~~
   \pi^{\F_I}_{s,t}=\begin{cases}
      id_{\F}, & \text{if}\ s,t\in I \\
      0, & \text{otherwise} \\
      \end{cases}.$$
Recall that a \textit{multiset} is a set in which elements may appear multiple times. The number of appearances of an element in a multiset is called its \textit{multiplicity.} 

\begin{defn}
A \textit{barcode} $\B$ is a multiset of intervals in $\R$ with finite multiplicities. The intervals in $\B$ are called \textit{bars}.
\end{defn}

\begin{thrm}[Structure theorem]\label{Theorem_Structure_PM}
Let $(V,\pi)$ be an upper semi-continuous, $q$-tame persistence module. There exists a unique barcode $\B(V)$ such that
$$V \cong \prod_{I\in \B(V)} (\F_I, \pi^{\F_I}) .$$
Moreover, each bar in $\B(V)$ is either of the form $(-\infty, b)$ for $-\infty<b\leq +\infty$ or $[a,b)$ for $-\infty <a<b\leq +\infty.$
\end{thrm}  

The structure theorem is a central result in topological persistence. In modern formulation, it has been proven in increasing generality starting from \cite{ELZ02,ZC05}, with precursors going all the way back to the work of Marston Morse. The most commonly cited version of the result was proven in \cite{WCB15}.  In the above stated form, the structure theorem appears as Theorem 1.3 in \cite{Schmahl22}.

\begin{rem} The standard definition of a barcode of a persistence module assumes a direct sum decomposition as in Theorem \ref{Theorem_Structure_PM}, instead of the direct product one.\footnote{Recall that, in contrast to the direct product, the direct sum $\oplus_{i\in \mathcal{I}}V_i$ of vector spaces $V_i$ contains only the vectors $v$ for which all but finitely many coordinates $v_i \in V_i$ are zero.} This subtlety has been elaborated in detail in \cite{Schmahl22}. We stick to the less common direct product convention since it is more suitable for our purposes. Let us also mention that our interest is in the bar-counting invariant $\mathcal{N}_\delta$ which is blind to this distinction, as has been explained in \cite[Section 2]{BPPPSS22}, see also \cite{CCBdS16}.
\end{rem}

\subsection{Barcodes of functions}\label{Subsection_Barcodes_Functions}

Let $X$ be a compact Hausdorff topological space, denote by $C^0(X)$ the space of continuous maps from $X$ to $\R$ and let $f\in C^0(X).$ We wish to further examine the persistence module $V_k(f)$ given by Example \ref{Example_Functions_Pers}.  To this end, we first note that this module is always upper semi-continuous, see \cite[Example 3.3]{Schmahl22}.  Thus, in order to obtain the interval decomposition given by the structure theorem, we need $V_k(f)$ to be $q$-tame. We will now prove that this condition is satisfied for a large class of spaces $X$ and $f\in C^0(X),$ which cover all persistence modules that we will consider. When $V_k(f)$ is $q$-tame, we denote by $\B_k(f)$ the barcode of $V_k(f)$ and if $V_k(f)$ is $q$-tame for all $k\geq 0$, we denote by $\B(f)= \cup_k \B_k(f).$

An \textit{$n$-dimensional box} is a subset of $\R^n$ of the form $[a_1,b_1]\times \ldots \times [a_n,b_n]$ for some intervals $[a_i,b_i]\subset \R.$ We will call an $n$-dimensional box of unspecified dimension $n$ simply a \textit{box}.

\begin{prop}\label{Proposition_q-tame}
Assume that there exist compact subsets $A_1,\ldots ,A_l$ of $X$ such that $X=A_1\cup \ldots \cup A_l$ and each $A_i$ is homeomorphic to a box. Then, for every $f\in C^0(X)$, $V_0(f)$ is $q$-tame.
\end{prop}

\begin{rem} Note that Proposition \ref{Proposition_q-tame} only applies to homology in degree zero, i.e. to $V_0(f).$ A similar result for all degrees $k\geq 0$ was proven in \cite[Lemma 4.9 ]{BPPPSS22}. However, in order to prove the result in all degrees, one needs to control the topology of the intersections of $A_i$, which is not needed in Proposition \ref{Proposition_q-tame}. In this regard, Proposition \ref{Proposition_q-tame} is a slight strengthening of \cite[Lemma 4.9 ]{BPPPSS22} in the special case $k=0.$ Secondly, let us point out that Proposition \ref{Proposition_q-tame} allows for $A_i$ to be homeomorphic to boxes of different dimensions $n_i.$ While we prove this general form of the proposition for the sake of completeness, in this paper we will apply Proposition \ref{Proposition_q-tame} only to boxes of equal dimensions.
\end{rem}

In order to prove Proposition \ref{Proposition_q-tame} we will need a couple of general facts about $V_k(f)$, which will be useful on their own.  Firstly, let us observe that if $\phi : X \to Y$ is a homeomorphism and $f\in C^0 (Y)$ then for all $k\geq 0$, 
\begin{equation}\label{Equation_PM_invariance}
V_k(f)\cong V_k(f\circ \phi).
\end{equation}
Indeed, since $\phi: \{ f\circ \phi \leq t \} \to \{ f\leq t \}$, the isomorphism is given by the homology-induced map $\phi_*: \check{H}_k(\{ f\circ \phi \leq t \}) \to \check{H}_k(\{ f\leq t \}).$ In other words, if $V_k(f)$ is $q$-tame, so is $V_k(f\circ \phi)$ and we have that for all $k\geq 0$
\begin{equation}\label{Equation_Barcode_invariance}
\mathcal{B}_k(f\circ \phi)=\mathcal{B}_k(f).
\end{equation}

Secondly, if $X$ is a box and $f\in C^0(X)$ then $V_k(f)$ is $q$-tame for all $k\geq 0.$ This fact has been elaborated in detail in \cite{BPPPSS22}. Now, (\ref{Equation_PM_invariance}) implies that if $X$ is homeomorphic to a box, then for every $f\in C^0(X)$, $V_k(f)$ is $q$-tame for all $k\geq 0.$

The last ingredient needed to prove Proposition \ref{Proposition_q-tame} is the following lemma.

\begin{lem}\label{Lemma_q-tame_MV}
Let $X$ be a compact Hausdorff topological space and let $A_1,A_2$ be compact subsets of $X$ such that $X=A_1 \cup A_2.$ Assume that $f\in C^0(X)$ is such that $V_0(f|_{A_1})$ and $V_0(f|_{A_2})$ are $q$-tame. Then $V_0(f)$ is also $q$-tame.
\end{lem}

\begin{proof}
From $X=A_1 \cup A_2$ it follows that for every $t\in \R$ it holds that
$$\{f\leq t \} = \{ f|_{A_1} \leq t\} \cup \{ f|_{A_2} \leq t \}.$$ 
Since $f$ is continuous and $X, A_1, A_2$ are compact, it follows that $\{f\leq t \},\{ f|_{A_1} \leq t\} ,\{ f|_{A_2} \leq t \}$ are compact as well. Thus, we may apply the Mayer-Vietoris sequence to these three sets which yields the exactness of 
$$\check{H}_0(\{f|_{A_1}\leq t \}) \oplus \check{H}_0(\{f|_{A_2}\leq t \}) \to \check{H}_0(\{f \leq t \}) \to 0.$$
Naturality of the Mayer-Vietoris sequence now yields an exact sequence of persistence modules
$$V_0(f|_{A_1})\oplus V_0(f|_{A_2})\to V_0(f) \to 0,$$
and the claim follows from \cite[Lemma 2.17]{BPPPSS22}.
\end{proof}

\begin{proof}[Proof of Proposition \ref{Proposition_q-tame}]
By (\ref{Equation_PM_invariance}) and the discussion succeeding it, we have that $V_0(f|_{Ai})$ is $q$-tame for $i=1,\ldots , l.$ Applying Lemma \ref{Lemma_q-tame_MV} $l-1$ times finishes the proof.
\end{proof}

Proposition \ref{Proposition_q-tame} provides a very general class of spaces $X$ for which $V_0(f)$ is $q$-tame for any $f\in C^0(X).$ The situation in degrees $k\geq 1$ is slightly more complicated. Nevertheless, a relatively large class of compact spaces $X$ have the property that $V_k(f)$ is $q$-tame for every $k\geq 0. $ For example, this is the case for any compact smooth manifold, with or without boundary, see \cite[Section 2]{BPPPSS22}. Assuming that $V_k(f)$ is $q$-tame for all $k\geq 0$, let us describe $\B(f).$

\begin{exm}[Example \ref{Example_Functions_Pers} continued]\label{Example_Functions_Pers2}

Let $X$ be compact, $f\in C^0(X)$ and assume that $V_k(f)$ is $q$-tame for all $k\geq 0.$ The function $f$ achieves global minimum and maximum denoted by $\min f, \max f.$ If $t<\min f$ we have that $\{f\leq t \} =\emptyset$ and thus $V_k(f)_t=\check{H}_k(\{f\leq t \})=0$ for all $k\geq 0.$ On the other hand, by definition, $\dim V_k(f)_t$ equals the number of bars in $\B_k(f)$ which contain $t.$ We conclude that there are no bars of the form $(-\infty,b)$ in $\B(f).$ Moreover, if $[a,b) \in \B(f)$ it follows that $a\geq \min f.$ 

Let us now take $t\geq \max f.$ We have that $\{ f\leq t \}=X$ and thus $V_k(f)_t=\check{H}_k(X)$ for all $k\geq 0.$ Moreover, for any $t\geq s \geq \max f$ it holds that $\pi_{s,t}=(id_X)_*=id_{\check{H}_k(X)}.$ Thus, no bar in $\B(f)$ has an endpoint greater than $\max f.$ We conclude that all the bars in $\B(f)$ are either of the form $[a,b)$ for $\min f \leq a <b \leq \max f$ or of the form $[a,+\infty)$ for $\min f \leq a \leq \max f.$

Infinite bars in $\B_k(f)$ correspond to homology classes in $\check{H}_k(X).$ Indeed, by the previous discussion, for any $t\geq \max f$, $\dim V_k(f)_t= \dim \check{H}_k(X).$ Thus, there are exactly $\dim \check{H}_k(X)$ bars of the form $[a,+\infty)$ in $\B_k(f).$ Finite bars are of different nature - they signify a topological measure of oscillations of $f.$ This point has been explained in detail in \cite{PPS19,PRSZ20,Stoji20, BPPPSS22,BPPSS23,Ginot24, Stoji24}.

\end{exm}

\subsection{The bar-counting invariant}

Let $V$ be an upper semi-continuous, $q$-tame persistence module and $\B(V)$ its barcode. For each $\delta \geq 0$, we denote by $\mathcal{N}_\delta(V)$ the number of bars, counting multiplicities, in $\B(V)$ of length strictly greater than $\delta.$ We would like to consider persistence modules for which $\mathcal{N}_\delta(V)$ is finite. For $\delta=0$ this boils down to $\B(V)$ being finite. This condition is too restrictive to cover all the modules which we wish to consider in our geometric set-up. Instead, we ask for $\mathcal{N}_\delta(V)$ to be finite for every $\delta>0.$ This class of modules will be sufficiently general for all our applications, while still having a rich enough theory, as we will now explain.  To this end, recall the following definition, introduced in \cite{BPPPSS22}.

\begin{defn}
A persistence module $V$ is called \textit{moderate} if it is upper semi-continuous, $q$-tame, $\mathcal{N}_\delta(V)$ is finite for all $\delta>0$ and $\B(V)$ does not contain bars of the form $(-\infty,b), -\infty < b \leq +\infty.$
\end{defn}

The justification for focusing only on moderate persistence modules comes from the following lemma.

\begin{lem}\label{Lemma_q-tame_Moderate}
Let $X$ be a compact Hausdorff topological space, $f\in C^0(X)$ and $k\geq 0.$ If $V_k(f)$ is $q$-tame then it is also moderate. 
\end{lem}

\begin{proof}
The fact that $V_k(f)$ is upper semi-continuous was discussed at the beginning of Subsection \ref{Subsection_Barcodes_Functions}, while the fact that $\B(V)$ does not contain bars of the form $(-\infty,b)$ was discussed in Example \ref{Example_Functions_Pers2}.  Lastly, $\mathcal{N}_\delta (V_k(f))$ is finite for all $\delta>0$ by \cite[Theorem 3.2]{BMMS24}. For the sake of completeness, let us give an independent short argument proving finiteness of $\mathcal{N}_\delta (V_k(f)).$

Assume that for some $\delta>0$, $\mathcal{N}_\delta(V_k(f))=+\infty.$ As explained in Example \ref{Example_Functions_Pers2} all the left endpoints of all bars in $\B(V_k(f))$ lie in $[\min f, \max f].$ Let $\{ [a_i, b_i) \}_{i\geq 1}$ be a sequence of bars of length greater than $\delta.$ Since all $a_i$ lie in $[\min f, \max f]$, there exists an accumulation point $a_\infty$ of the sequence $\{a_i \}_{i\geq 1}.$ It follows that infinitely many bars from the sequence $\{ [a_i, b_i) \}_{i\geq 1}$ contain the interval $[a_\infty + \frac{\delta}{3}, a_\infty + \frac{2\delta}{3}].$ This contradicts $q$-tameness of $V_k(f)$, since the rank of $\pi_{a_\infty + \frac{\delta}{3}, a_\infty + \frac{2\delta}{3}}$ is infinite.
\end{proof}

For $X$ and $f$ as in Lemma \ref{Lemma_q-tame_Moderate}, we denote by $\mathcal{N}_{k,\delta}(f)=\mathcal{N}_{\delta}(V_k(f))$ and by $\mathcal{N}_{\delta}(f)=\sum_{k\geq 0}\mathcal{N}_{k,\delta}(f).$

\begin{defn}
A compact Hausdorff topological space $X$ is \textit{$k$-moderate} if for every $f\in C^0(X)$, $V_k(f)$ is moderate.  A space which is $k$-moderate for all $k\geq 0$ is called just moderate.
\end{defn}

\begin{rem} In \cite[Definition 4.7]{BPPPSS22} the notion of a tame family is introduced.  Lemma \ref{Lemma_q-tame_Moderate} implies that a compact Hausdorff topological space $X$ is moderate if and only if a 1-element family $\{X\}$ is tame.  We prefer to use the term moderate since both definitions can be applied to non-compact spaces, in which case the two notions differ.
\end{rem}

Our main object of interest is $\mathcal{N}_{0,\delta}(f)$ for $\delta>0.$ Together, Proposition \ref{Proposition_q-tame} and Lemma \ref{Lemma_q-tame_Moderate} provide a plethora of 0-moderate spaces. For example, every compact topological manifold, with or without boundary, is a 0-moderate space. We have thus shown that $\mathcal{N}_{0,\delta}(f)$ is well-defined and finite for a large class of functions, which includes all functions considered in this paper. As a next step, we wish to show two properties of $\mathcal{N}_{0,\delta}$, subadditivity and stability, which will be the driving forces behind our arguments.

\begin{thrm}[Subadditivity theorem - \cite{BPPPSS22}]\label{Therem_Subadditivity_Deg_0} Let $X=A_1 \cup \ldots \cup A_l$, where each $A_i$ is 0-moderate. Then $X$ is also 0-moderate and for every $f\in C^0(X)$ it holds that
\begin{equation}\label{Equation_Subadditivity_Thm}
\mathcal{N}_{0,\delta} (f) \leq \sum_{i=1}^l \mathcal{N}_{0,\delta} (f|_{A_{i}}).
\end{equation}
\end{thrm}

\begin{proof}
The space $X$ is compact as a finite union of compacts and it is 0-moderate by Lemmas \ref{Lemma_q-tame_MV} and \ref{Lemma_q-tame_Moderate}. The proof of (\ref{Equation_Subadditivity_Thm}) is given in \cite[Remark 4.10]{BPPPSS22}.
\end{proof}

\begin{rem}\label{Remark_Subadditivity_Higher_Deg}
The subadditivity theorem also holds in degrees $k\geq 1$, albeit with worse control of the parameter $\delta$ on the left-hand side of (\ref{Equation_Subadditivity_Thm}) and with additional terms coming from the intersections of $A_i$ on the right-hand side of (\ref{Equation_Subadditivity_Thm}), see \cite[Proposition 4.8]{BPPPSS22}.
\end{rem}

Let $k\geq 0$ be an integer and $X$ be a $k$-moderate space. For $f,h \in C^0(X)$ let $d_{C^0}(f,h)= \max_{x\in X} |f(x)- h(x)|.$

\begin{prop}[Stability of $\mathcal{N}_\delta$]\label{Proposition_BarCount_Stability}
For all $\delta\geq 0$ and all $d_{C^0}(f,h) \leq \varepsilon \leq \frac{\delta}{2}$ it holds that
$$\mathcal{N}_{k,\delta}(f) \leq \mathcal{N}_{k,\delta-2\varepsilon}(h). $$
\end{prop}

Proposition \ref{Proposition_BarCount_Stability} is a simple consequence of the more general \textit{stability theorem} for barcodes. This theorem was discovered in \cite{CEH07} and has since been the flagship result of topological persistence.  The modern formulation and the proof of the stability theorem can be found in \cite{BL15,CdSGO16,CCBdS16}. Let us mention that both the subadditivity theorem and the stability theorem have formulations which apply to general persistence modules.  The versions stated here are special cases of these results obtained by applying them to persistence modules of the form $V_k(f).$
 
Lastly, we will need a duality result about $\mathcal{N}_{k,\delta}.$ To this end, let $\mathcal{N}^{fin}_{k,\delta}(f)=\mathcal{N}_{k,\delta}(f)-\dim \check{H}_k(X)$ denote the number of finite bars of length greater than $\delta$ in $\B_k(f).$ The following result is a direct corollary of \cite[Proposition 2.19]{BPPPSS22}, see also \cite[Remark 2.11]{BPPPSS22} and \cite[Proposition 2.12]{BPPPSS22}.
 
\begin{prop}\label{Proposition_Duality_Barcodes}
Assume that $X$ is a closed, smooth, orientable, $n$-dimensional manifold and let $f\in C^0(X).$ For every degree $0\leq k \leq n-1$ and all $\delta>0$ it holds that
$$\mathcal{N}^{fin}_{k,\delta}(f)=\mathcal{N}^{fin}_{n-k-1,\delta}(-f).$$
\end{prop}

\subsection{Coarse counts}\label{Subsection_Coarse_Classical}

In the introduction, we defined the coarse nodal count and the coarse count of zeros.  The following result connects these notions to the bar counting invariant.

\begin{prop}\label{Proposition_Counts_Comparison}
Let $X$ be a moderate space, $l\geq 1$ an integer and $F:X\to \R^l$ a continuous map. For all $0<\varepsilon<\delta$ and all degrees $k$ it holds that
\begin{equation}\label{Inequality_Zeros_Bar}
z_{k,\delta}(F)\leq \mathcal{N}_{k,\delta}(|F|),
\end{equation}
as well as
\begin{equation}\label{Inequality_Nodal_Bar}
m_{k,\delta}(F)\leq \mathcal{N}_{k,\delta - \varepsilon}(-|F|).
\end{equation}
\end{prop}
\begin{proof}
To prove (\ref{Inequality_Zeros_Bar}), we notice that $\{F=0 \} = \{ |F|\leq 0 \}$ and thus
$$z_{k,\delta}(F) = \rank \left( \check{H}_k(\{   |F| \leq 0 \}) \to \check{H}_k(\{  |F|\leq \delta\}) \right)= 
\rank \pi_{0,\delta}^{V_k(|F|)}.$$
By the definition of a barcode $\rank \pi_{0,\delta}^{V_k(|F|)}$ equals the number of bars in $\B_k(|F|)$ which contain the interval $[0,\delta].$ Since all these bars have lengths greater than $\delta$,  (\ref{Inequality_Zeros_Bar}) follows.

To prove (\ref{Inequality_Nodal_Bar}), we notice that $\{F\neq 0 \} = \{|F|>0\} = \{-|F|<0 \}$ and thus
$$m_{k,\delta} (F)= \rank \left( \check{H}_k(\{  -|F|< -\delta \}) \to \check{H}_k(\{  -|F|<0\}) \right).$$
To estimate this rank, notice that for every $0<\varepsilon < \delta$, the relevant map factors as
$$\check{H}_k(\{  -|F|< -\delta \}) \to \check{H}_k(\{  -|F|\leq -\delta\}) \to \check{H}_k(\{  -|F|\leq -\varepsilon \}) \to \check{H}_k(\{  -|F|<0\}),$$
and thus
$$m_{k,\delta}(F)\leq \rank \left( \check{H}_k(\{  -|F|\leq -\delta \}) \to \check{H}_k(\{  -|F|\leq -\varepsilon \}) \right)= \rank  \pi_{-\delta,-\varepsilon}^{V_k(-|F|)}.$$
By definition, $\rank  \pi_{-\delta,-\varepsilon}^{V_k(-|F|)}$ is equal to the number of bars in $\B_k(-|F|)$ which contain the interval $[-\delta, -\varepsilon]$ and (\ref{Inequality_Nodal_Bar}) follows.
\end{proof}

\begin{rem}
In fact, a stronger form of (\ref{Inequality_Nodal_Bar}) holds. Namely, for all $\delta>0$ and all $k\geq 0$
$$m_{k,\delta} (F)\leq \mathcal{N}_{k,\delta}(-|F|).$$
The proof of this inequality can be extracted from \cite[Section 2]{BPPPSS22}. 
\end{rem}

The proof of Proposition \ref{Proposition_Counts_Comparison} shows that $z_{k,\delta}$ only counts certain bars of length greater than $\delta$ in $\B_{k}(|F|)$, namely those which contain $[0,\delta]$ (and similarly for $m_{k,\delta}$ and $\B_{k}(-|F|)$). In this regard, $z_{k,\delta}$ and $m_{k,\delta}$ can be considered to be localized at $[0,\delta]$ and $[-\delta,0]$ versions of $\mathcal{N}_{k,\delta}.$ The following example shows that the difference between $z_{k,\delta},m_{k,\delta}$ and $\mathcal{N}_{k,\delta}$ can in general be large.

\begin{exm}\label{Example_Discrepancy}
Let $S^1=\R/ 2\pi \Z$ and let $f_n : S^1 \to \R$ be given by $f_n(x)=\sin (nx)+100, n\geq 1.$ The barcode of $f_n$ is given by
$$\B_0(f_n)=\{ [99,+\infty), [99,101)\times (n-1) \} ,~\B_1(f_n)=\{ [101,+\infty) \}.$$
In particular, for any $0<\delta<2$ we have that $\mathcal{N}_{0,\delta}(f_n)=n-1$, while $m_{0,\delta}(f_n)=1,z_{0,\delta}(f_n)=0.$ 
\end{exm}

\section{Stratified groups and nilmanifolds}\label{Section_Stratified}

In this section we review basic facts about stratified Lie groups and their nilmanifolds. The material which we cover is classical, see \cite{Rag76,FS82,BLU_Book,LeDonne25} for an extensive treatment.

\subsection{Basics}\label{Subsection_Carnot_Nilmanifolds}

Let $\frak{g}$ be a finite-dimensional Lie algebra over $\R.$\footnote{Unless otherwise specified, all Lie groups and Lie algebras are assumed to be finite-dimensional and defined over $\R.$ Similarly, all Lie groups are assumed to be connected.} Denote by $\frak{g}^{(1)}=\frak{g}, \frak{g}^{(i+1)}=Span_\R ([\frak{g},\frak{g}^{(i)}])$ for all $i\geq 1.$ Recall that $\frak{g}$ is called \textit{nilpotent} if there exists $i_0$ such that $\frak{g}^{(i_0)}=\{ 0\}.$ If $\frak{g}$ is the Lie algebra of a Lie group (meaning the Lie algebra of left-invariant vector fields), then $\frak{g}$ being nilpotent is equivalent to the group being nilpotent.  Let $\G=(G,\cdot)$ be a nilpotent Lie group. A subgroup $\Gamma <\G$ is called a \textit{lattice} if it is discrete and the quotient space $\Gamma \backslash \G = \{ \Gamma \cdot g ~|~ g\in G\}$ is compact. Given a lattice $\Gamma<\G$,  $N=\Gamma \backslash \G$ is a closed manifold and $\pi: G \to N$ is a smooth covering. Such manifolds are called \textit{nilmanifolds}. Our main focus will be on a class of nilmanifolds which are obtained as quotients of stratified groups.  We now recall the relevant definitions and properties.

A \textit{stratification} on a Lie algebra $\frak{g}$ is a direct sum decomposition
$$\frak{g}=V_1 \oplus \ldots \oplus V_s,$$
such that $V_s\neq \{0 \}$ and for all $j\geq 1$,
\begin{equation}\label{Equation_Stratification_Def}
V_{j+1}=Span_\R ([V_1,V_j]),
\end{equation}
where $V_j=\{0 \}$ for $j\geq s+1.$ If a Lie algebra admits a stratification it is called \textit{stratifiable} and if a stratification is fixed, it is called \textit{stratified}. Given two different stratifications of $\frak{g}$
$$\frak{g}=V_1 \oplus \ldots \oplus V_s=V'_1 \oplus \ldots \oplus V'_{s'},$$
there exists a Lie algebra automorphism $\Phi$ of $\frak{g}$ such that $\Phi(V_i)=V'_i$ for all $i\geq 1.$ In particular $s=s'$ and $V_i \cong V'_i$ for all $i.$ This allows us to introduce the following notions.

\begin{defn}\label{Definition_Stratified_Lie_Algebra}
Given a stratifiable Lie algebra $\frak{g}$, the integer $s$ is called the \textit{step} of $\frak{g}$, while $Q=\sum\limits_{i=1}^{s} i \cdot \dim V_i$ is called the \textit{homogeneous dimension} of $\frak{g}.$ The \textit{weights} $\omega_1,\ldots , \omega_n$ of $\frak{g}$ are given by 
$$(\omega_1,\ldots , \omega_n)=(1,\ldots, 1,2, \ldots ,2, \ldots , s,\ldots,s),$$
where each $i$ is repeated $\dim V_i$ times.  Note that $Q=\omega_1+ \ldots + \omega_n.$
\end{defn}

\begin{defn}\label{Definiton_Stratified_Group}
A simply connected Lie group $\G$ whose Lie algebra $\frak{g}$ is stratified is called a \textit{stratified Lie group}.  The step and the homogeneous dimension of $\G$ are the step and the homogeneous dimension of $\frak{g}. $
\end{defn}

\begin{rem}\label{Remark_Stratified}
Note that in Definition \ref{Definiton_Stratified_Group}, we ask for $\frak{g}$ to be stratified rather than merely stratifiable.  This is equivalent to making a choice of $V_1$, due to (\ref{Equation_Stratification_Def}). 
\end{rem}

In order to do analysis on a stratified group $\G$, we need to fix a measure on its Borel $\sigma$-algebra.  To this end, recall that $\G$ can be equipped with a left Haar measure, which is unique up to a multiplication by a positive constant. Moreover, stratified groups are unimodular, i.e. any left Haar measure $\mu$ is also a right Haar measure, which makes it bi-invariant. A way of explicitly constructing such a measure is as the push-forward of a Lebesgue measure by the exponential map. More precisely, if $\exp : \frak{g} \to G$ is the exponential map and $\mu_0$ a Lebesgue measure on $\frak{g}$, then $\exp_* \mu_0$ is a bi-invariant Haar measure on $\G.$ In particular, every bi-invariant Haar measure $\mu$ on $\G$ satisfies
\begin{equation}\label{Equation_Haar_Exp}
\mu = c_\mu \cdot exp_* \mu_0
\end{equation}
for some constant $c_\mu.$

\begin{exm}\label{Example_Heisenberg}
Let $\H_3(\R)$ be a subgroup of $\rm{GL}_3(\R)$ given by
$$\H_3(\R)= \left\lbrace \begin{bmatrix}
    1 & x & z \\
    0 & 1 & y \\
    0 & 0 & 1 
\end{bmatrix} ~ \Bigg\vert~x,y,z \in \R \right\rbrace .$$
One readily checks that $\H_3(\R)$ is a closed subgroup of $\rm{GL}_3(\R)$, hence a Lie group, as well as that it is nilpotent.  It is called the \textit{3-dimensional real Heisenberg group}. It is diffeomorphic to $\R^3$,  the global coordinates being given by $(x,y,z).$ Its Lie algebra, denoted by $\frak{h}_3$, has a basis given by the following left-invariant vector fields
$$X_1=\partial_x,~X_2=\partial_y+x\partial_z,~X_3=\partial_z,$$
which satisfy
$$[X_1, X_2]=X_3,~[X_1,X_3]=[X_2,X_3]=0.$$
Thus,  $V_1= Span_\R(X_1,X_2), V_2=Span_\R(X_3)$ gives a stratification of $\frak{h}_3$ with step 2 and weights $\omega_1=\omega_2=1, \omega_3=2.$ The homogeneous dimension of $\H_3(\R)$ is equal to 4.  Let $\H_3(\Z)$ be a subgroup of $\H_3(\R)$ given by matrices with integer entries $x,y,z.$ This is a discrete subgroup of $\H_3(\R)$, called the \textit{3-dimensional integer Heisenberg group}.  In order to check that it is a lattice in $\H_3(\R)$, we will describe the quotient $\H_3(\Z) \backslash \H_3(\R).$ To this end,  for $x\in \R$ let $\lfloor x \rfloor$ be the integer part of $x$ and $\{ x\}= x- \lfloor x \rfloor$ the fractional part of $x.$ Since
$$\begin{bmatrix}
    1 & - \lfloor x \rfloor & - \lfloor z -y \lfloor x \rfloor \rfloor \\
    0 & 1 & - \lfloor y \rfloor \\
    0 & 0 & 1 
\end{bmatrix} 
\begin{bmatrix}
    1 & x & z \\
    0 & 1 & y \\
    0 & 0 & 1 
\end{bmatrix}
=
\begin{bmatrix}
    1 & \{ x \} & \{ z - y\lfloor x \rfloor \} \\
    0 & 1 & \{ y \} \\
    0 & 0 & 1 
\end{bmatrix} ,$$
we conclude that $[0,1]^3$ is a fundamental domain for the action of $\H_3(\Z)$ on $\H_3(\R).$ A straightforward analysis of the action on this domain shows that $\H_3(\Z) \backslash \H_3(\R)$ is diffeomorphic to a mapping torus of a Dehn twist on $\T^2.$  Analogous definitions of the $(2n+1)$-dimensional Heisenberg groups and corresponding nilmanifolds can be given for any $n\geq 1$, see \cite[Example 9.3.7]{LeDonne25} for details.
\end{exm}

\subsection{Exponential coordinates}\label{Subsection_Exp_Coordinates}

Recall that two simply connected Lie groups are isomorphic if and only if their Lie algebras are isomorphic. Moreover, for every Lie algebra $\frak{g}$ there exists a simply connected Lie group whose Lie algebra is isomorphic to $\frak{g}.$ In this regard, stratified Lie groups correspond to stratified Lie algebras. To make this correspondence explicit, let $\G$ be a stratified group, $\frak{g}$ its Lie algebra and $exp:\frak{g}\to G$ the exponential map (here as before, we identify $\frak{g}$ with the space of left-invariant vector fields on $G$).  One may show that $exp$ is a diffeomorphism.  Thus, fixing a basis of $\frak{g}$ defines global coordinates on $G.$ We will repeatedly use such coordinates in our arguments. Moreover, for the sake of simplicity, we will work with bases of special type, called the \textit{Carnot bases}\footnote{We adopted this terminology from \cite{LeDonne25}.} of $\frak{g}.$ We now recall the relevant definitions.

Let $\frak{g}=V_1\oplus \ldots \oplus V_s$ be a stratification of $\frak{g}.$ Let $m_0=0, m_1=m=\dim V_1,m_2=\dim (V_1 \oplus V_2), \ldots ,  m_s=\dim (V_1\oplus \ldots \oplus V_s).$ 

\begin{defn}
A basis $X_1,\ldots , X_n$ is called a \textit{Carnot basis} of $\frak{g}$ if 
\begin{enumerate}
\item For all $1\leq i \leq s$, $X_{m_{i-1}+1},\ldots , X_{m_i}$ is a basis of $V_i.$ In particular, $X_j \in V_{\omega_j}$ for all $j.$
\item For every $j\geq m_1+1$ there exist $k_j,l_j$ such that $X_{k_j}\in V_1, X_{l_j} \in V_{\omega_{j-1}}$ and $X_j = [X_{k_j}, X_{l_j}].$ 
\end{enumerate}
\end{defn}

Every stratified Lie algebra has a Carnot basis. Furthermore, any basis $X_1,\ldots , X_m$ of $V_1$ can be extended to a Carnot basis $X_1,\ldots , X_n.$ This is done by iteratively choosing a basis of $V_{j+1}$ from the brackets of already chosen basis elements of $V_1$ and $V_j.$

From now on, let us fix a Carnot basis $X_1,\ldots, X_n.$ The \textit{homogeneous $\infty$-norm} on $\G$ with respect to the basis $X_1,\ldots , X_n$ is given by
$$| exp(x_1 X_1 +\ldots x_n X_n) |_\infty = \max_i |x_i|^\frac{1}{\omega_i}.$$
For $g\in G$,  $r>0$, denote by

\begin{equation}\label{Equation_Homogeneous_Ball}
B_\infty(g,r)= \{ h\in G~|~ |g^{-1} h|_\infty \leq r \}.
\end{equation}
One readily checks that
\begin{equation}\label{Equation_Translate_Ball}
B_\infty(g,r)=g\cdot B_\infty(id,r).
\end{equation}

Let $\mu$ be a bi-invariant Haar measure on $\G.$ We wish to compute $\mu (B_\infty(g,r)).$ To this end, notice that 
$$\mu(B_\infty(g,r))=\mu (g \cdot B_\infty(id,r))=\mu (B_\infty(id,r)),$$
by (\ref{Equation_Translate_Ball}) and the bi-invariance of $\mu.$ Let
$$\Bx_{(\omega_1,\ldots ,\omega_n)}(r)= \{ (x_1,\ldots ,x_n)\in \R^n ~|~ (\forall i)~ |x_i|\leq r^{\omega_i} \}.$$
By definition
\begin{equation}\label{Equation_Box_Box}
B_\infty(id,r) = exp (\Bx_{(\omega_1,\ldots ,\omega_n)}(r)),
\end{equation}
which together with  (\ref{Equation_Haar_Exp}) yields
\begin{equation}\label{Equation_Volume_of_Balls}
\mu (B_\infty(g,r))= c_\mu \cdot \mu_0 (\Bx_{(\omega_1,\ldots ,\omega_n)}(r))=a r^Q,
\end{equation}
where $a=c_\mu \cdot \mu_0 (\Bx_{(\omega_1,\ldots ,\omega_n)}(1)).$

\begin{rem}\label{Remark_Triangle_Homogeneous}
Intuitively, we think of $B_\infty(g,r)$ as a closed ball with center $g$ and radius $r$ with respect to $|\cdot |_\infty.$ However, this is not a metric ball in the usual sense since, in general,  homogeneous norms do not satisfy the triangle inequality. Nevertheless, for every stratified group $\G$ and a Carnot basis  $X_1, \ldots , X_n$ there exists a constant $C$ such that 
\begin{equation}\label{Equation_Almost_Triangle}
|gh|_\infty  \leq C (|g|_\infty + |h|_\infty).
\end{equation}
\end{rem}

The relation between the Lie group structure on $\G$ and the Lie algebra structure on $\frak{g}$ under $exp$ is expressed by the Baker-Campbell-Hausdorff formula. We will not discuss this formula in detail, as it will not be needed in the paper. Nevertheless, we will make use of the following corollary of the Baker-Campbell-Hausdorff formula (for a proof, see Remark 1.4.4 and  Proposition 2.2.22 in \cite{BLU_Book}).

\begin{prop}\label{Proposition_Polynomial_Coordinates}
Let $X_1,\ldots, X_n$ be a Carnot basis of $\frak{g}$ and let $g,h\in \G$ be given by 
$$g=exp(u_1X_1+\ldots + u_nX_n), h=exp(v_1X_1+\ldots + v_nX_n),$$
and their product by 
$$g \cdot h = exp(w_1 X_1+\ldots + w_n X_n).$$
For all $1\leq i \leq n$ it holds that
$$w_i=u_i+v_i+P_i(u_1,\ldots , u_{i-1}, v_1,\ldots , v_{i-1}),$$
where $P_i$ is a polynomial on $\R^{2i-2}$ of degree at most $\omega_i.$
\end{prop}

\section{Sub-Riemannian geometry}\label{Section_Sub-Riemannian}

This section covers elements of sub-Riemannian geometry used throughout the text.  For a detailed treatment of the subject we refer the reader to \cite{SubRiemannian96,Mont02,ABB19,LeDonne25}.

\subsection{Sub-Riemannian manifolds}\label{Subsection_SR_Basic}

Let $M$ be a smooth manifold and denote by $C^\infty(M;TM)$ the set of smooth sections of $TM$, i.e. smooth vector fields on $M.$ Commutators of vector fields render $C^\infty(M;TM)$ a Lie algebra. Given a subset $\mathcal{V}\subset C^\infty(M;TM)$, we denote by $\Lie(\mathcal{V})$ the minimal Lie subalgebra of $C^\infty(M;TM)$ which contains $\mathcal{V}.$ We say that $\mathcal{V}$ satisfies \textit{H\"{o}rmander's condition} if for all $p\in M$, $\Lie(\mathcal{V})_p=T_p M,$ where $\Lie(\mathcal{V})_p= \{ X(p)~|~X \in \Lie(\mathcal{V}) \}.$

Recall that a \textit{regular distribution} on $M$ is a smooth subbundle of $TM.$ A regular distribution $\xi$ is called \textit{bracket-generating} if $C^\infty(M;\xi)$ satisfies H\"{o}rmander's condition. A \textit{sub-Riemannian metric} on a bracket-generating distribution $\xi$ is a smooth family of scalar products of $\xi.$

\begin{defn}\label{Definition_SubRiemannian_Mfd}
A sub-Riemannian manifold is a triple $(M,\xi,g)$, $M$ being a smooth manifold, $\xi$ a bracket-generating distribution and $g$ a sub-Riemannian metric on $\xi.$ 
\end{defn}

Given a sub-Riemannian manifold $(M,\xi,g),$ $\xi$ is called the \textit{horizontal distribution.} Similarly, an absolutely continuous curve $\gamma: [a,b] \to M$ is called \textit{horizontal} if it is tangent to $\xi$ almost everywhere. In this case 
$$\Length (\gamma)= \int_{a}^{b} \|\dot{\gamma}(t) \|_g dt.$$

\begin{defn}
The Carnot-Carath\'{e}odory distance on a sub-Riemannian manifold is given by
$$d_{cc}(p,q) = \inf \{ \Length(\gamma)~|~ \gamma \text{ is a horizontal curve from } p \text{ to } q\}.$$
\end{defn}

A priori, it is not clear that there exists a horizontal curve between any two points on $M$, i.e.  $d_{cc}$ might not always be finite. This issue is resolved by the following theorem.

\begin{thrm}[Chow-Rashevskii]
If $M$ is connected, there exists a horizontal curve between any two points on $M.$ In particular, $d_{cc}$ is a genuine metric. Moreover, the topology induced by $d_{cc}$ is the manifold topology on $M.$
\end{thrm}

\subsection{Carnot groups and sub-Riemannian nilmanifolds}\label{Subsection_Carnot_Sub-Riemannian}

Let $\G$ be a stratified group and $\frak{g}=V_1\oplus \ldots \oplus V_s$ its stratified Lie algebra, which we identify with the space of left-invariant vector fields. As we mentioned in Remark \ref{Remark_Stratified}, fixing a stratification of $\frak{g}$ is equivalent to fixing $V_1.$ Furthermore, this choice defines a bracket-generating distribution on $G$ given by
\begin{equation}\label{Equation_Bracket_Generating_Distribution}
\xi_g= \{ X_g~|~ X \in V_1 \},~g\in G.
\end{equation}
Consequently, if a nilmanifold $N=\Gamma \backslash \G$ is a quotient of a stratified group, it is equipped with a bracket-generating distribution which at point $q\in N$ is given by 
$$
\xi^N_q=\pi_*(\xi_{\tilde{q}}) \text{ for any } \tilde{q}\in \pi^{-1}(q).
$$

In order to define a sub-Riemannian structure on $(G,\xi)$, we need to define a family of scalar products on $\xi.$ To this end, let us introduce the following notion. 

\begin{defn}\label{Definiton_Carnot_Group}
A \textit{Carnot group} is a stratified Lie group $\G$ together with a scalar product $\langle ,\rangle_{V_1}$ on $V_1.$
\end{defn}

Given a Carnot group $\G$, we define a sub-Riemannian structure on $(G,\xi)$ by extending $\langle,  \rangle_{V_1}$ by left-invariance.  More precisely, for $g\in G$ and $X,Y \in T_g G$, let $\tilde{X},\tilde{Y}\in \frak{g}$ be the unique left-invariant vector fields such that $\tilde{X}_g=X,\tilde{Y}_g=Y.$ Then
\begin{equation}\label{Equation_Scalar_Product}
\langle X,Y \rangle_g = \langle \tilde{X}, \tilde{Y} \rangle_{V_1} .
\end{equation}
This product descends to lattice-quotients, yielding a family of scalar products $\langle,\rangle^N$ on the distribution $\xi^N$ on any nilmanifold $N=\Gamma \backslash \G.$ Slightly abusing the notation, we abbreviate $\xi^N$ to $\xi$ and $\langle,\rangle^N$ to $\langle,\rangle.$ The triples $(G,\xi,\langle,\rangle)$ and $(N,\xi,\langle,\rangle)$ are sub-Riemannian manifolds according to Definition \ref{Definition_SubRiemannian_Mfd}.  

From the viewpoint of sub-Riemannian geometry, Carnot-Carath\'{e}odory distance is the natural metric to consider on a Carnot group $\G. $ Next, we wish to relate the notion of size introduced by an $\infty$-homogeneous norm on $\G$ to that defined by the Carnot-Carath\'{e}odory distance. To this end, let $X_1,\ldots , X_n$ be a Carnot basis of $\frak{g}$ and let $B_\infty(\cdot,\cdot)$ be given by (\ref{Equation_Homogeneous_Ball}).  We denote by
$$B_{cc}(g,r)= \{ h\in G~|~ d_{cc}(g,h)<r \},$$
the open ball with respect to $d_{cc}.$ The following result and its proof are given as Theorem 11.2.3 in \cite{LeDonne25}.

\begin{thrm}[Ball-box theorem]\label{Thm_Ball_Box_Carnot}
There exists $C>1$ such that for all  $g\in G$, $r> 0$ it holds that
$$B_{cc}(g,r)\subset B_\infty(g,Cr) \subset B_{cc}(g,C^2r).$$
\end{thrm}  

\begin{rem}\label{Remark_Constant_Subriemannian}
A priori, the constant $C$ in Theorem \ref{Thm_Ball_Box_Carnot} depends on the choice of $X_1,\ldots, X_n.$ However, by the definition of the Carnot basis, once we fix the first $m$ vector fields $X_1,\ldots , X_m$, there are only finitely many choices of $X_{m+1}, \ldots , X_n.$ Thus, we may take $C$ which only depends on the choice of $X_1, \ldots , X_m.$ Furthermore, if we assume that $X_1,\ldots , X_m$ is an orthonormal basis of $V_1$, we may take $C$ which only depends on $\G.$ Namely, there is a compact family of choices of orthonormal bases of $V_1$ (parametrized by $O(m)$) and we may take $C$ to be the maximal constant in this family.
\end{rem}

We will make use of the following simple corollary of the ball-box theorem.

\begin{cor}\label{Corollary_Compact_CC}
A subset $K\subset G$ is compact if and only if it is closed and bounded with respect to $d_{cc}.$
\end{cor}
\begin{proof}
By the Chow-Rashevskii theorem, $K$ is closed in the manifold topology on $G$ if and only if it is closed in the topology induced by $d_{cc}.$ Let $\Psi_{id}: \R^n \to G$ be the exponential coordinates with respect to any Carnot basis.  Since $\Psi_{id}$ is a diffeomorphism, $K$ is compact if and only if $ \Psi_{id}^{-1}(K)\subset \R^n$ is compact, which is equivalent to $ \Psi_{id}^{-1}(K)$ being closed and bounded with respect to the standard distance on $\R^n.$ On the other hand,  $ \Psi_{id}^{-1}(K)$ is closed if and only if $K$ is closed and  $\Psi_{id}^{-1}(K)$ is bounded if and only if $K\subset B_\infty(id,r)$ for some finite $r. $ The last condition is equivalent to $K$ being bounded in $d_{cc}$ by Theorem \ref{Thm_Ball_Box_Carnot}.
\end{proof}

\subsection{Equiregular structures}\label{Subsection_Equiregular_Structures}

Carnot groups and nilmanifolds obtained as their quotients are examples of a wider class of sub-Riemannian manifolds, called \textit{equiregular} sub-Riemannian manifolds, which we will now introduce. Let $(M,\xi)$ be a connected smooth manifold with a bracket generating distribution. For $i\geq 0$ an integer and $p\in M$ let
$$\xi^{(i)}_p = Span_\R (\{ [X_1, [X_2,\ldots , [X_{j-1},X_j]\ldots ]_p~|~ 1\leq j \leq i, ~X_1,\ldots , X_j \in C^\infty(M; \xi))\},$$
where $\xi^{(0)}=\{0 \}$ is the zero subbundle of $TM$ and $\xi^{(1)}=\xi.$

\begin{defn}
The point $p\in M$ is called \textit{regular} if for every $i\geq 1$ there exists a neighbourhood of $p$ in $M$ on which the map $q\to \dim (\xi^{(i)}_q)$ is constant. The distribution $\xi$ is called \textit{equiregular} if every point in $M$ is regular.
\end{defn}

Since $M$ is connected, the equiregularity of $\xi$ is equivalent to $q\to \dim (\xi^{(i)}_q)$ being a constant function for every $i\geq 1.$ In other words, this condition means that each $\xi^{(i)}$, $i\geq 1$ is a regular distribution on $M.$ Furthermore, since $\xi$ is bracket generating, it follows that the ranks of $\xi^{(i)}$ are strictly increasing until they reach $n=\rank (TM).$ More precisely,  there exists an integer $s$ such that
\begin{equation}\label{Equation_Equiregular_Inclusions}
\{0\}=\xi^{(0)} \varsubsetneq\xi^{(1)} \varsubsetneq \xi^{(2)} \varsubsetneq \ldots \varsubsetneq \xi^{(s)} =TM.
\end{equation}

\begin{defn}\label{Definition_Equiregular_Step_Q}
The unique integer $s$ defined by (\ref{Equation_Equiregular_Inclusions}) is called the \textit{step} of the equiregular distribution $\xi.$ The \textit{homogeneous dimension of} $(M,\xi)$ is given by $Q=\sum\limits_{i=1}^s i \cdot (\rank \xi^{(i)}-\rank \xi^{(i-1)}).$
\end{defn}

Definition \ref{Definition_Equiregular_Step_Q} generalizes the analogous notions for stratified groups and their nilmanifolds, see Definitions \ref{Definition_Stratified_Lie_Algebra} and \ref{Definiton_Stratified_Group}.  Namely, $(G,\xi)$ and $(N,\xi)$ defined in Subsection \ref{Subsection_Carnot_Sub-Riemannian} are examples of equiregular distributions. Indeed, one readily checks that for these distributions
$$\xi^{(i)}_p= \{ X_p~|~ X \in V_1 \oplus \ldots \oplus V_i \},$$
and thus $\rank \xi^{(i)}-\rank \xi^{(i-1)}=\dim V_i.$

Manifolds with equiregular distributions encompass many other geometrically interesting classes of spaces. For example, every contact distribution is equiregular with step two.

\begin{defn}
A sub-Riemannian manifold $(M,\xi,g)$ is called \textit{equiregular} if the distribution $\xi$ is equiregular.
\end{defn}

Continuing the above discussion,  we note that Carnot groups and their corresponding nilmanifolds are examples of equiregular sub-Riemannian manifolds.

\subsection{Generating vector fields}\label{Subsection_Generating_VFs}

Up until this point, we mostly discussed sub-Riemannian manifolds from a coordinate free perspective. In order to study Sobolev spaces and subelliptic operators, we will make use of a more concrete description of sub-Riemannian structures, based on generating vector fields. To this end, let $M$ be a smooth manifold, $\xi$ a bracket-generating distribution on $M$ and recall from the introduction that a finite set of vector fields $X=\{ X_1, \ldots , X_m \}$ is said to be \textit{generating for} $\xi$ if for every $q\in M$, $\xi_q= Span_\R (X_1(q),\ldots , X_m(q)).$ Such a set $X$ satisfies the H\"{o}rmander's condition. Indeed, one readily checks that for all $i\geq 1$ and all $q\in M$
$$\xi^{(i)}_q = Span_\R (\{ [X_{l_1}, [X_{l_2},\ldots , [X_{l_{j-1}},X_{l_j}]\ldots ]_q~|~ 1\leq j \leq i, ~l_1,\ldots , l_j \in \{ 1,\ldots , m \} \}).$$

A generating set of vector fields $X=\{ X_1, \ldots , X_m \}$ induces a sub-Riemannian structure $g_X$ on $\xi,$ defined via its norm: 
$$\| Y_q \|_{g_X}=\inf \{ (a_1^2+\ldots +a_m^2)^{\frac{1}{2}}~|~ Y_q=a_1 X_1(q)+\ldots + a_mX_m(q) \},$$
for every $q\in M, Y_q\in \xi_q.$ 

\begin{exm}\label{Example_Gen_Car}
Let $\G$ be a stratified group, $\frak{g}=V_1 \oplus \ldots \oplus V_s$ a stratification of its Lie algebra and $\xi$ a bracket-generating distribution defined by (\ref{Equation_Bracket_Generating_Distribution}). Any basis $X_1,\ldots, X_m$ of $V_1$ is a generating set of vector fields for $\xi.$ Moreover, this generating set is minimal in the sense that $m= \rank \xi.$ Assuming $\G$ is a Carnot group, i.e. a scalar product is given on $V_1$, we may take $X_1,\ldots, X_m$ to be an orthonormal basis of $V_1.$ In this case, the sub-Riemannian structure defined by (\ref{Equation_Scalar_Product}) coincides with $g_X.$ Since all the notions are local,  we obtain a generating set on any nilmanifold $N=\Gamma \backslash \G$ by projecting via $\pi : G \to N.$
\end{exm}

It turns out that every sub-Riemannian structure is of the type we just described.

\begin{thrm}\label{Theorem_Generating_VF}
Let $(M,\xi,g)$ be a smooth sub-Riemannian manifold. There exists a finite set $X=\{X_1, \ldots,X_m\}$ of smooth vector fields on $M$ which are generating for $\xi$ and which satisfy $g_X=g.$
\end{thrm}

Theorem \ref{Theorem_Generating_VF} is a special case of the material presented in \cite[Subsection 3.1.4]{ABB19}.  Namely, there is a more general notion of a sub-Riemannian structure and any finite collection of vector fields which satisfies H\"{o}rmander's condition corresponds to such a structure. In this regard, sub-Riemannian geometry can be studied purely from the viewpoint of vector fields satisfying H\"{o}rmander's condition. This perspective is commonly taken when studying hypoelliptic operators.

\subsection{Anisotropic Sobolev spaces}\label{Subsection_Anisotropic_Sobolev}

Let $M$ be a smooth manifold with a Borel measure $\mu$ defined by a smooth positive density. Let $X=\{ X_1, \ldots , X_m \}$ be a finite set of vector fields which satisfies H\"{o}rmander's condition.  For $f\in C^\infty(M)$, $K$ a Borel subset of $M$, $p\geq 1$ and $k\geq 1$ an integer,  we denote by
$$\| X^k f \|_{L^p(K)}= \sum_{1\leq i_1 , \ldots ,   i_k \leq m} \| X_{i_1}\circ \ldots \circ X_{i_k} f \|_{L^p(K)},$$
and define the \textit{anisotropic $(k,p)$-Sobolev norm of $f$ on $K$} as
$$
 \|f \|_{W^{k,p}_X(K)} = \sum_{i=1}^k \| X^i f \|_{L^p(K)} + \| f\|_{L^p(K)}.
$$
We will omit $K$ from the notation when considering functions only defined on $K.$ The \textit{anisotropic Sobolev space}  $W^{k,p}_X(M)$ is the set of all functions  $f\in L^p(M)$ which have all the weak derivatives up to order $k$ with respect to $\{ X_i\}_{i=1,\ldots , m}$ in $L^p(M).$ The norms $\|\cdot \|_{W^{k,p}_X} $ extend to $ W^{k,p}_X(M).$ One may develop the theory of anisotropic Sobolev spaces, partially in analogy to the case of the usual Sobolev spaces. This is a subject of ongoing research, which was initiated by the seminal work of H\"{o}rmander - \cite{Hormander67}. Anisotropic theory is more subtle than the standard one and any attempt at a comprehensive exposition exceeds the scope of this text. Thus, we will only state the results we will use in the exact generality in which they will be applied. 

As we mentioned in Subsection \ref{Subsection_Generating_VFs}, vector fields satisfying H\"{o}rmander's condition are more general than the sub-Riemannian structures we wish to consider.  Namely, we will always work with vector fields which generate an equiregular structure.  In this setting, we wish to compare anisotropic Sobolev norm of a function to its usual Sobolev norm.

\begin{prop}[Theorem 13 in \cite{RS76}]\label{Proposition_Sobolev_Sobolev}
Let $(M,g)$ be a closed Riemannian manifold and assume that $X_1, \ldots , X_m $ generate an equiregular distribution $\xi$ on $M$ of step $s.$ For every $p>1$ and $k\geq 1$ an integer there exists a constant $C$ such that for every $f\in C^\infty(M)$ it holds that
$$\|f \|_{W^{k,p}} \leq C \|f \|_{W^{sk,p}_X} .$$
\end{prop}

\begin{rem}
In Proposition \ref{Proposition_Sobolev_Sobolev}, the measure $\mu$, as well as the usual Sobolev norms, are induced by the Riemannian metric $g.$ Consequently, the constant $C$ depends on $(M,g),k,p$ and the vector fields $X_1, \ldots , X_m.$
\end{rem}

Let us narrow our focus even more and consider a stratified group $\G.$ Let $\frak{g}=V_1 \oplus \ldots \oplus V_s$ be a stratification of its Lie algebra, $Q$ its homogeneous dimension and $\xi$ a bracket-generating distribution defined by (\ref{Equation_Bracket_Generating_Distribution}). In this case we will take $\mu$ to be a Haar measure and $X_1,\ldots ,X_m$ a basis of $V_1.$ Let $\Gamma < \G$ be a lattice and $N$ the corresponding nilmanifold equipped with the measure $\nu$ induced by $\mu.$ As before, slightly abusing the notation, we identify $X_i$ and $\pi_* X_i.$ In this case, the anisotropic Sobolev norm defined in this subsection coincides with the definition given by (\ref{eq:anisotropic_Sobolev}). The following result is an anisotropic version of the classical Sobolev inequality.

\begin{thrm}\label{Theorem_Basic_Sobolev}
Let $p\geq 1$ and $k$ an integer such that $kp>Q.$ For every $f\in C^{\infty}(N) $ it holds that
\begin{equation}\label{Equation_Sobolev_inequality}
\|f\|_{L^\infty}  \leq C  \|f \|_{W^{k,p}_X},
\end{equation}
where $C$ depends only on $\G,\Gamma,\mu,k,p$ and the vector fields $X_1,\ldots,X_m.$ 
\end{thrm}

The inequality (\ref{Equation_Sobolev_inequality}) for compactly supported functions on $\G$ is classical. It can be found, for example, in \cite[Chapter IV, Theorem 5.8]{Robinson91} or \cite[Corollary IV.7.4 (ii)]{VSCC92} (our setting corresponds to $d=D=Q$ in the notation of \cite{VSCC92}). Theorem \ref{Theorem_Basic_Sobolev} follows from this inequality by fixing a cover $\{U_i\}$ of $N$, along with a subordinate partition of unity $\{\phi_i \}$, and then applying the local estimate to each $\phi_i \cdot f.$

\begin{cor}\label{Corollary_Basic_Sobolev}
Let $p\geq 1, l\geq 1 $ an integer and $k$ an integer such that $kp>Q.$ For every $F=(f_1,\ldots ,f_l)$,  $f_1,\ldots, f_l \in C^{\infty}(N)$ it holds that
$$\max |F| \leq C \|F \|_{W^{k,p}_X},$$
where $C$ depends only on $\G,\Gamma,\mu,l,k,p$ and the vector fields $X_1,\ldots,X_m.$ 
\end{cor}
\begin{proof}
Using Theorem \ref{Theorem_Basic_Sobolev} we estimate
$$\max |F| = \max \sqrt{\sum_{i=1}^l f_i^2}  \leq \sqrt{\sum_{i=1}^l \| f_i\|_{L^\infty}^2} \leq \sqrt{l} \max_{1\leq i \leq l} \|f_i \|_{L^\infty} \leq C\sqrt{l} \max_{1\leq i \leq l} \|f_i \|_{W^{k,p}_X}.$$
Applying the estimate  $ \max_{1\leq i \leq l} \|f_i \|_{W^{k,p}_X} \leq  \|F \|_{W^{k,p}_X}  $ finishes the proof.
\end{proof}

\subsection{Sub-Laplacians}

Let $M$ be a closed manifold, $\mu$ a measure on $M$ with a smooth density and let $X_1,\ldots , X_m$ be vector fields on $M$ which satisfy H\"{o}rmander's condition.  The \textit{H\"{o}rmander's sub-Laplacian} associated to $(\{X_i\},\mu)$ is the operator
$$P_{\{X_i\},\mu}=X_1^* X_1 + \ldots + X_m^* X_m,$$ 
where $X_i^*$ denotes the formal $L^2$-dual of $X_i$ given by $X_i^*=-X_i-\divergence_\mu X_i.$ The properties of this operator have been extensively studied starting from the pioneering work \cite{Hormander67}, most notably in \cite{RS76}.  In particular, it was shown that $P_{\{X_i\},\mu}$ is hypoelliptic and satisfies sub-elliptic estimates. Building on this body of work, one may develop the spectral theory of $P_{\{X_i\},\mu}$ by analogy with the case of the Laplace-Beltrami operator, see, for example,  \cite[Section 2]{EL23}, \cite[Subsection 3.1]{CH26} and references therein. More precisely, this operator is symmetric and non-negative on $C^\infty (M)$ and it admits a Friedrichs extension with a compact resolvent. This further implies that it has a discrete spectrum which tends to infinity, as well as an orthonormal $L^2$-basis of eigenfunctions.  Furthermore,  hypoellipticity of $P_{\{X_i\},\mu}$ implies that the eigenfunctions of $P_{\{X_i\},\mu}$ are $C^\infty$-smooth.

We will make use of the following result, see \cite[Theorem 2.10.18.]{Street14} for a proof.

\begin{thrm}\label{Theorem_Subelliptic_Estimates}
Let $M$ be a closed manifold,  $\mu$ a measure on $M$ with a smooth density, $\{X_i\}$ a finite set of vector fields on $M$ which satisfies H\"{o}rmander's condition and $P=P_{\{X_i\},\mu}$ the associated  H\"{o}rmander's sub-Laplacian. For all integers $k\geq 1$ and all $f\in C^\infty(M)$, it holds that
$$\| f \|_{W^{k,2}_X} \leq C \|(1+P)^\frac{k}{2} f\|_{L^2},$$
where $C$ only depends on $M, \{X_i\},\mu$ and $k.$
\end{thrm}

Assume now that $(M,\xi,g)$ is a closed, equiregular sub-Riemannian manifold. To each measure $\mu$ with a smooth density, we may associate an operator $P_\mu$ in a way which closely resembles the usual coordinate-free definition of the Laplace-Beltrami operator on a Riemannian manifold. Namely, given $f\in C^\infty(M)$, we define its horizontal gradient  $\nabla^\xi f $ as the unique vector field in $C^\infty(M ; \xi)$ such that $g(\nabla^\xi f, V)=df(V)$ for all $V \in C^\infty(M ; \xi).$ Let
$$P_\mu = - \divergence_\mu \circ \nabla^\xi.$$
We will call this operator the \textit{sub-Riemannian Laplacian associated to $\mu$.} It is an example of a H\"{o}rmander's sub-Laplacian. To see this, we use the correspondence between sub-Riemannian structures and vector fields given by Theorem \ref{Theorem_Generating_VF}. More precisely, let $X=\{X_1,\ldots , X_m \}$ be a generating set for $\xi$ such that $g_X=g.$ One computes that\footnote{This argument can be found in \cite{LL22}, see \cite[Remark 1.30]{LL22} and references therein}
$$\nabla^\xi f = \sum_{i=1}^m X_i(f) X_i,$$
and thus
\begin{equation}\label{Equation:Sublaplacians_Equality}
 P_\mu (f) =- \divergence_\mu \left(\sum_{i=1}^m X_i(f) X_i \right)= - \sum_{i=1}^m (X_i(f) \divergence_\mu X_i + X_i^2(f) ) = P_{X,\mu}(f).
\end{equation}

As with the other sub-Riemannian notions, when considering a Carnot group $\G$ or its nilmanifold $\Gamma \backslash \G$, we may take $X_1,\ldots , X_m$ to be an orthonormal basis of the first layer of the stratification of $\frak{g}$ and $\mu$ a Haar measure on $G$, see Example \ref{Example_Gen_Car}.

\section{Polynomials on stratified groups}

The central idea behind our method is local polynomial approximation of functions in anisotropic Sobolev spaces. In this section we develop the necessary technical background for applying this idea in the setting of stratified groups.

\subsection{Basics}

As before, let $\G$ be a stratified group, $\frak{g}$ its Lie algebra with weights $\omega_1,\ldots , \omega_n$ and $exp:\frak{g}\to G$ the exponential map. A \textit{polynomial on $\G$} is a function $P:G\to \R$ such that $P\circ exp:\frak{g} \to \R$ is a real polynomial. Polynomials on $\G$ form a finitely-generated algebra. A generating set of this algebra can be given by composing the dual basis to a Carnot basis with the exponential map. More precisely, let $X_1,\ldots , X_n$ be a Carnot basis of $\frak{g}$ and $\xi_1,\ldots , \xi_n:\frak{g} \to \R$ its dual basis (explicitly $\xi_i(x_1 X_1 +\ldots + x_nX_n)=x_i$).  The functions $\{\eta_i=\xi_i \circ exp^{-1} \}_{1\leq i\leq n}$ generate the polynomial algebra on $\G$, i.e. every polynomial $P$ can be uniquely written as
\begin{equation}\label{Equation_Polynomial_Expression}
P= \sum_{i_1,\ldots ,i_n \in \Z_{\geq 0}} a_{i_1,\ldots , i_n} \eta_1^{i_1} \ldots \eta_{n}^{i_n},
\end{equation}
where only finitely many $a_{i_1,\ldots , i_n}$ are non-zero.  The \textit{homogeneous degree} of $\eta_1^{i_1} \ldots \eta_{n}^{i_n}$ is given by
$$\deg_h(\eta_1^{i_1} \ldots \eta_{n}^{i_n})=\sum_{j=1}^n i_j \omega_j,$$
while the \textit{homogeneous degree} of $P$ given by (\ref{Equation_Polynomial_Expression}) is defined as
$$\deg_h(P)= \max \{ \deg_h(\eta_1^{i_1} \ldots \eta_{n}^{i_n})~|~a_{i_1,\ldots , i_n} \neq 0 \}.$$ We denote by $\mathcal{P}_d(\G)$ the set of polynomials of homogeneous degree at most $d$ on $\G.$ Similarly, we denote by $\mathcal{P}_d(\R^n)$ the set of polynomials of (usual) degree at most $d$ on $\R^n.$\footnote{Note that this notation is consistent with the fact that $(\R,+)$ is an abelian stratified group whose Lie algebra has a trivial, 1-layer, stratification.}

We wish to consider polynomials in different charts given by the exponential coordinates. To this end, recall that the \textit{exponential coordinates with respect to $X_1,\ldots, X_n$ centered at $g$} are given by the map $\Psi_g: \R^n \to G$,
$$\Psi_g(x_1,\ldots,x_n)= g \cdot exp (x_1 X_1 + \ldots + x_n X_n).$$
Note that $\Psi_g$ is a diffeomorphism such that $\Psi_g(0)=g.$ We will need the following lemma.

\begin{lem}\label{Lemma_Polynomial_Degree_Change}
If $P \in \mathcal{P}_d(\G)$ then for all $g\in G$,  $P\circ \Psi_g \in \mathcal{P}_{d}(\R^n).$ 
\end{lem}
\begin{proof}
Let $g= exp(a_1 X_1 + \ldots +a_n X_n).$ Proposition \ref{Proposition_Polynomial_Coordinates} implies that there exist $P_i \in \mathcal{P}_{\omega_i}(\R^{2i-2})$ for $1\leq i \leq n$ such that for all $(x_1,\ldots, x_n) \in \R^n$
$$g\cdot exp (x_1 X_1+\ldots +x_n X_n)=exp(y_1 X_1+\ldots +y_n X_n),$$
where $y_i=a_i+x_i+P_i(a_1,\ldots, a_{i-1},x_1,\ldots,x_{i-1}).$ By definition of $\Psi_g$, we have that
$$P\circ \Psi_g(x_1,\ldots ,x_n) = P(g\cdot exp (x_1 X_1+\ldots +x_n X_n))=P(exp(y_1 X_1+\ldots +y_n X_n))=$$
$$= P\circ exp \left( (a_1+x_1)X_1+  \ldots + (a_n+x_n+ P_n(a_1,\ldots, a_{n-1},x_1,\ldots,x_{n-1})) X_n \right),$$
and the claim follows since each $P_i$ has degree at most $\omega_i.$
\end{proof}

\subsection{Barcode of a polynomial on a box}

Let $\G$ be a stratified group, $\frak{g}$ its Lie algebra, $X_1, \ldots , X_n$ a Carnot basis of $\frak{g}.$ Recall that $B_\infty(g,r)$ denotes the homogeneous $\infty$-ball with respect to $X_1, \ldots , X_n$, while $\mathcal{N}_0(f)$ denotes the total number of bars in the barcode of $f$ (in all degrees).

\begin{prop}\label{Proposition_Barcode_Bound} 
For any degree $d\geq 0$,  $P\in \mathcal{P}_d(\G)$ and any $r\geq 0, g\in G$ it holds that
$$ \mathcal{N}_{0}( P|_{B_\infty(g,r)} ) \leq \frac{(d+1)^n+1}{2}.$$
\end{prop}

\begin{proof}

Recall from (\ref{Equation_Box_Box}) that $B_\infty(id,r)=exp (\Bx_{(\omega_1,\ldots ,\omega_N)}(r)).$ Thus
$$B_\infty(g,r)=g\cdot B_\infty(id,r)=g \cdot exp (\Bx_{(\omega_1,\ldots ,\omega_N)}(r)) = \Psi_g (\Bx_{(\omega_1,\ldots ,\omega_N)}(r)).$$
Applying (\ref{Equation_Barcode_invariance}), we obtain
$$\mathcal{N}_0(P|_{B_\infty(g,r)} )=\mathcal{N}_0(P\circ \Psi_g|_{\Bx_{(\omega_1,\ldots, \omega_n)}(r)} ).$$
On the other hand, by Lemma \ref{Lemma_Polynomial_Degree_Change}, $P\circ \Psi_g\in \mathcal{P}_d (\R^n).$ Thus, the statement reduces to the case of a polynomial on a box in $\R^n,$ which has been treated in \cite{BPPPSS22}.  Indeed, by Proposition 4.12 and Remark 4.13 in \cite{BPPPSS22}, we have that 
$$\mathcal{N}_{0}(P\circ \Psi_g|_{\Bx_{(\omega_1,\ldots, \omega_n)}(r)} )  \leq  \frac{(d+1)^n+1}{2},$$
and the claim follows.
\end{proof}

\subsection{Morrey-Sobolev inequality}

Let $\G$ be a Carnot group, $X_1,\ldots X_n$ a Carnot basis of $\frak{g}=V_1\oplus \ldots \oplus V_s$, such that $X_1,\ldots, X_m$ is a basis of $V_1.$ The following result is proven in \cite[Theorem 5.11]{LW00}, see also \cite[Theorem 5.11]{LW04} and \cite[Appendix A]{Irene26}.

\begin{thrm}[\cite{LW00,LW04}]\label{Theorem_Polynomial_Approximation}
Let $p\geq 1$ and $k$ an integer such that $kp>Q.$  There exists $C>0$ such that for every ball $B=B_{cc}(g,r)$ and every $f\in C^\infty(B)$ there exists $P\in  \mathcal{P}_{k-1}(\G)$ such that
$$d_{C^0}(f,P|_B)\leq C r^{k-\frac{Q}{p}} \| X^k f \|_{L^p}  . $$
\end{thrm}

\subsection{Square roots of polynomials}

Theorem \ref{Thm:Sobolev_Nilmanifolds} applies to the barcodes of $\pm |F|. $ In order to prove this theorem, we will approximate $F$ by a polynomial map, which in turn means that $|F|$ is approximated by the square root of a positive polynomial, while $-|F|$ is approximated by a negative of the square root of a positive polynomial. With this in mind, we wish to explore the properties of square roots of polynomials and their negatives. To this end, as before let $\G$ be a stratified group, $\frak{g}$ its Lie algebra, $X_1, \ldots , X_n$ a Carnot basis of $\frak{g}$ and $B_\infty(g,r)$ the homogeneous $\infty$-ball with respect to $X_1, \ldots , X_n.$ For every integer $d\geq 0$, let
$$\mathcal{S}_d(\G)= \{ \pm \sqrt{P} ~|~ P \in \mathcal{P}_{2d}(\G), P\geq 0 \}.$$
The following is a direct corollary of Proposition \ref{Proposition_Barcode_Bound}.
\begin{cor}\label{Corollary_Barcode_Bound}
For any degree $d\geq 0$,  $S\in \mathcal{S}_d(\G)$ and any $r\geq 0, g\in G$ it holds that
$$ \mathcal{N}_{0}( S|_{B_\infty(g,r)} ) \leq \frac{(2d+1)^n+1}{2}.$$
\end{cor}
\begin{proof}
We first treat the case $S=\sqrt{P}$ for $P\in \mathcal{P}_{2d}(\G).$ Since $P\geq 0$ it follows that
$$\{ P \leq t \}= \{ S^2 \leq t \} = \begin{cases}
      \{ S\leq \sqrt{t} \}, & t\geq 0 \\
      \emptyset, & t<0 \\
      \end{cases}.$$
Applying $\check{H}_*(\cdot)$ we get that
$$\mathcal{B}(P)= \{ [a^2,b^2)~|~[a,b)\in \mathcal{B}(S),~a<b<+\infty \}\cup \{ [c^2,+\infty)~|~[c,+\infty)\in \mathcal{B}(S) \}.$$
In particular the barcodes $\mathcal{B}(P)$ and $\mathcal{B}(S)$ have the same number of bars. Since $P\in \mathcal{P}_{2d}(\G)$ the claim follows from Proposition \ref{Proposition_Barcode_Bound}. 

Assume now that $S=-\sqrt{P}$ for $P\in \mathcal{P}_{2d}(\G).$ Since $-P\leq 0$ it follows that
$$\{ -P \leq t \}= \{ -S^2 \leq t \} = \begin{cases}
      B_\infty(g,r) & t>0 \\
      \{ S\leq -\sqrt{-t} \}, & t\leq 0 \\
      \end{cases}.$$ 
Similarly as before, applying $\check{H}_*(\cdot)$ yields
$$\mathcal{B}(-P)= \{ [-a^2,-b^2)~|~[a,b)\in \mathcal{B}(S),~a<b<+\infty \}\cup \{ [-c^2,+\infty)~|~[c,+\infty)\in \mathcal{B}(S) \},$$
and the barcodes $\mathcal{B}(-P)$ and $\mathcal{B}(S)$ have the same number of bars. Since $-P\in \mathcal{P}_{2d}(\G)$ the claim again follows from Proposition \ref{Proposition_Barcode_Bound}. 
\end{proof}

Let us now fix a scalar product on the first layer of stratification of $\frak{g}$, rendering $\G$ a Carnot group. The following is a direct corollary of the Morrey-Sobolev inequality.

\begin{cor}\label{Corollary_Polynomial_Approximation}
Let $p\geq 1, l\geq 1$ an integer and $k$ an integer such that $kp>Q.$  There exists $C>0$ such that for every ball $B=B_{cc}(g,r)$ and every $F=(f_1,\ldots , f_l), f_1,\ldots , f_l \in C^\infty(B)$ there exists $S\in  \mathcal{S}_{k-1}(\G)$ such that
$$d_{C^0}(|F|,S|_B)\leq C r^{k-\frac{Q}{p}} \| F \|_{W^{k,p}_X}  . $$
\end{cor}
\begin{proof}
Let $P_1,\ldots , P_l \in \mathcal{P}_{k-1}(\G)$ be polynomials which approximate $f_1,\ldots ,f_l$ given by Theorem \ref{Theorem_Polynomial_Approximation} and take $S=\sqrt{\sum_{i=1}^l P_i^2}.$ By definition $S \in \mathcal{S}_{k-1}(\G).$ We estimate
$$d_{C^0}(|F|,S|_B)= \max_B \left(\sqrt{\sum_{i=1}^l f_i^2} - \sqrt{\sum_{i=1}^l P_i^2}\right) \leq \max_B \sqrt{\sum_{i=1}^l (f_i - P_i)^2}\leq $$
$$\leq \sqrt{\sum_{i=1}^l \max_B (f_i - P_i)^2} \leq \sqrt{l} \max_{1\leq i \leq l} d_{C^0}(f_i, P_i|_B)\leq \sqrt{l} C' r^{k-\frac{Q}{p}} \max_{1\leq i \leq l} \| X^k f_i \|_{L^p},$$
where $C'$ is the constant given by Theorem \ref{Theorem_Polynomial_Approximation}. Since $\max_{1\leq i \leq l} \| X^k f_i \|_{L^p}\leq \| F \|_{W^{k,p}_X}$, setting $C=\sqrt{l}C'$ finishes the proof.
\end{proof}

\section{Proof of Theorem \ref{Thm:Sobolev_Nilmanifolds}}

Let $X_1,\ldots, X_n$ be a Carnot basis of $\frak{g}$ obtained by extending $X_1,  \ldots , X_{m}$ and let us fix an auxiliary scalar product on the first layer of stratification of $\frak{g}$, rendering $\G$ a Carnot group. The constants which appear throughout this section depend on the choices of $X_{m+1},\ldots, X_n$ and the scalar product. In turn, our proof of Theorem \ref{Thm:Sobolev_Nilmanifolds} yields an upper bound with the constant $C$ depending on these choices. In order to obtain the optimal constant, one should take the infimum of $C$ over all choices of $X_{m+1},\ldots, X_n$ and the scalar product. Recall that $B_\infty(\cdot,\cdot)$ denotes the closed homogeneous $\infty$-ball given by (\ref{Equation_Homogeneous_Ball}), while $B_{cc}(\cdot, \cdot)$ denotes the open ball in the Carnot-Carath\'{e}odory metric, see Subsection \ref{Subsection_SR_Basic}. We start by proving two auxiliary lemmas. 

\begin{lem}\label{Lemma_Covering_Threshold}
Let $K\subset G$ be compact. There exists $r_0>0$ such that for all $g\in K$ and all $r<r_0$, $\pi : B_{cc}(g,  r) \to N$ is an embedding.
\end{lem}
\begin{proof}
Let 
$$K'=\{ g\in G ~|~ d_{cc}(g,K)\leq 1 \}.$$
Corollary \ref{Corollary_Compact_CC} implies that $K'$ is compact.  Since $\pi:G\to N$ is a covering, for every $g\in K'$ there exists $r(g)$ such that $\pi : B_{cc}(g,r(g))\to N$ is an embedding. Let  $\{B_{cc}(g_i,r(g_i)) \}_{i=1,\ldots , q}$ be a finite subcover of the open cover $\{B_{cc}(g,r(g)) \}_{g\in K'}$ of $K'.$ Considering $(K',d_{cc})$ as a compact metric space itself, $\{B_{cc}(g_i,r(g_i)) \cap K' \}_{i=1,\ldots , q}$ is an open cover and, as such, has a Lebesgue number. Let $r_{Leb}$ be the Lebesgue number of this cover and $r_0 = \min (\frac{r_{Leb}}{2}, \frac{1}{2}).$ Using the definition of the Lebesgue number and the fact that $r_0 <1$,  we have that for every $g\in K$ there exists $1\leq i \leq q$ such that
$$B_{cc}(g,r_0) = B_{cc}(g,r_0) \cap K'\subset B_{cc}(g_i,r(g_i)) \cap K'\subset B_{cc}(g_i,r(g_i)).$$
Since $\pi$ is an embedding on all $B_{cc}(g_i,r(g_i))$ the claim follows.
\end{proof}

\begin{lem}\label{Lemma_Uniscale_Covering}
Let $K\subset G$ be compact.  There exists $c$ such that for all $r>0$, there exists $q\leq \max(\frac{c}{r^Q},1)$ and  $g_1,\ldots , g_q \in K$ such that $K\subset \bigcup_{i=1}^q B_{\infty}(g_i,r).$
\end{lem}

Lemma \ref{Lemma_Uniscale_Covering} is a consequence of a covering lemma in the style of Vitali and Wiener. The general statement of the Vitali-Wiener lemma, as well as its proof, are given in \cite[Chapter 1.F]{FS82}. We will only need the following special case.

\begin{lem}\label{Lemma_Vitali_Wiener} There exists a constant $\alpha<1$ such that for every compact set $K\subset G$ and every $r>0$ there exist $g_1,\ldots ,g_q \in K$ such that
\begin{enumerate}
\item $K \subset \bigcup\limits_{i=1}^q B_\infty(g_i,r)$;
\item $B_\infty(g_i, \alpha r)$, $i=1,\ldots , q$ are disjoint.
\end{enumerate} 
\end{lem}

\begin{proof}[Proof of Lemma \ref{Lemma_Uniscale_Covering}]
Since $K$ is bounded, there exists $R_0$ and $g\in K$ such that $K\subset B_\infty(g,R_0).$ Thus, if $r\geq R_0$, we may take $q=1.$ We are left to prove the statement for $r<R_0.$ Let $g_1,\ldots , g_q \in K$ be given by Lemma \ref{Lemma_Vitali_Wiener}. By property (1) we have that $K \subset \bigcup\limits_{i=1}^q B_\infty(g_i,r)$ and we need to prove that $q\leq \frac{c}{r^Q}.$ To this end, let $K'$ be the $R_0$-thickening of $K$ with respect to $|\cdot |_\infty$ given by
$$K'= \bigcup\limits_{g\in K} B_\infty(g,R_0).$$
We claim that $K'$ is compact. By Corollary \ref{Corollary_Compact_CC} it is enough to show that it is closed and bounded in $d_{cc}.$ It is closed because $| \cdot |_\infty$ is continuous and $K$ is compact. On the other hand, it is bounded in $| \cdot |_\infty$ by (\ref{Equation_Almost_Triangle}) and thus also in $d_{cc}$ by Theorem \ref{Thm_Ball_Box_Carnot}.

Since $K'$ is compact we have that it is measurable and $\mu(K')<+\infty$ by the definition of a Haar measure. Furthermore,
$$\bigcup\limits_{i=1}^q B_\infty(g_i,\alpha r)  \subset \bigcup\limits_{i=1}^q B_\infty(g_i,R_0) \subset K',$$
and thus property (2) in Lemma \ref{Lemma_Vitali_Wiener} implies that
$$\sum_{i=1}^q \mu(B_\infty(g_i,\alpha r))  \leq \mu(K') . $$
Combining this inequality and (\ref{Equation_Volume_of_Balls}) proves $q\leq \frac{c}{r^Q},$ with $c=\frac{\mu(K')}{a\alpha^Q}.$
\end{proof}

We are now in a position to prove Theorem \ref{Thm:Sobolev_Nilmanifolds}.

\begin{proof}[Proof of Theorem \ref{Thm:Sobolev_Nilmanifolds}]

Let $C_0$ be the constant given by Corollary \ref{Corollary_Basic_Sobolev} so that
$$\max |F| - \min |F| \leq \max |F| \leq C_0   \|F\|_{W^{k,p}_X}.$$
If $C_0   \|F \|_{W^{k,p}_X} < \delta$ we have that 
$$\max(-|F|) - \min (-|F|) = \max |F| - \min |F| <\delta.$$
Thus, the only bars of length greater than $\delta$ in $\B(\pm|F|)$ are infinite bars. In other words, $\mathcal{N}_{d,\delta}(\pm|F|)=b_d(N)$ and Theorem \ref{Thm:Sobolev_Nilmanifolds} follows. Hence, in the rest of the proof we assume that $\frac{\delta}{C_0   \|F \|_{W^{k,p}_X}} \leq 1.$

We will first prove the case $d=0$ and then treat the case $d=n-1$ at the very end. Let $C_1$ be the constant given by Corollary \ref{Corollary_Polynomial_Approximation} and $C_2$ the constant given by the ball-box theorem (Theorem \ref{Thm_Ball_Box_Carnot}). Denote by $\tilde{F}=F\circ \pi :G\to \R^l$ the lift of $F.$ Let $K\subset G$ be a compact set such that $\pi(K)=N$ and let $r_0$ be the threshold given by Lemma \ref{Lemma_Covering_Threshold}. We wish to cover $K$ by homogeneous balls $B_\infty(g_i, r)$, $g_i \in G$, on which $|F|$ has a good approximation by the square root of a polynomial. To this end, we first choose a scale parameter $r.$ Let 
$$r= C_3 \left( \frac{\delta}{2C_1  \|F \|_{W^{k,p}_X}} \right)^{\frac{p}{kp-Q}} ,$$
where
$$C_3= \min \left(\frac{1}{2C_2} , \frac{r_0}{2C_2} \left( \frac{2C_1}{C_0} \right)^{\frac{p}{kp-Q}} \right).$$
Since $\frac{\delta}{C_0   \|F \|_{W^{k,p}_X}} \leq 1,$ we have that
\begin{equation}\label{Equation_threshold}
r < \min \left( \frac{1}{C_2} \left( \frac{\delta}{2C_1  \|F \|_{W^{k,p}_X}} \right)^{\frac{p}{kp-Q}} , \frac{r_0}{C_2} \right).
\end{equation}
Let $g_1,\ldots , g_q \in G$ be given by Lemma \ref{Lemma_Uniscale_Covering} for this choice of $K$ and $r.$ It follows that
$$
q\leq \max\left(\frac{c}{r^Q},1 \right)=\max \left( \frac{c (2C_1)^{\frac{Qp}{kp-Q}}}{C_3^Q} \left( \frac{  \|F \|_{W^{k,p}_X}}{\delta} \right)^{\frac{Qp}{kp-Q}},1\right).
$$
Since by assumption $\frac{\delta}{C_0   \|F \|_{W^{k,p}_X}} \leq 1,$ we conclude that
\begin{equation}\label{Equation_Number_of_Balls}
q \leq C'  \left( \frac{  \|F \|_{W^{k,p}_X}}{\delta} \right)^{\frac{Qp}{kp-Q}} ,
\end{equation}
where $C'=\max \left(\frac{c (2C_1)^{\frac{Qp}{kp-Q}}}{C_3^Q }, C_0^{\frac{Qp}{kp-Q}} \right).$

We wish to prove that for all $1\leq i \leq q$ it holds that
\begin{equation}\label{Equation_Delta_Box}
\mathcal{N}_{0,\delta}(\pm|F|_{\pi(B_\infty(g_i,r) )}|)\leq \frac{(2k-1)^n+1}{2}. 
\end{equation}
This estimate together with (\ref{Equation_Number_of_Balls}) readily yields the proof of the theorem. Indeed, by the subadditivity theorem (Theorem \ref{Therem_Subadditivity_Deg_0})
$$
\mathcal{N}_{0,\delta}(\pm|F|) \leq \sum_{i=1}^q \mathcal{N}_{0,\delta}(\pm|F|_{\pi(B_\infty(g_i,r) )}|) \leq q \cdot \frac{(2k-1)^n+1}{2} \leq C \left( \frac{  \|F \|_{W^{k,p}_X}}{\delta} \right)^{\frac{Qp}{kp-Q}}, 
$$
where $C=C' \frac{(2k-1)^n+1}{2}.$\footnote{Note that the inequality which we obtain does not include the $+b_0(N)$ term on the right-hand side. This is not surprising since we are working under the assumption that $\frac{  \|F \|_{W^{k,p}_X}}{\delta}\geq \frac{1}{C_0}$ and thus any constant term can be absorbed in $C.$}

What is left is to prove (\ref{Equation_Delta_Box}). To this end, notice that, by (\ref{Equation_threshold}), $C_2r<r_0$ and thus Lemma \ref{Lemma_Covering_Threshold} implies that $\pi: B_{cc}(g_i,  C_2 r) \to N$ is an embedding for all $i.$ On the other hand, the ball-box theorem implies that $B_\infty(g_i,r)\subset B_{cc}(g_i,C_2 r)$, and hence $\pi$ is a homeomorphism on all $B_\infty(g_i,r).$ Equation (\ref{Equation_Barcode_invariance}) implies that
$$\mathcal{B}(\pm|F|_{\pi(B_\infty(g_i,r) )}|)= \mathcal{B}((\pm|F|\circ \pi)|_{B_\infty(g_i,r) }))=\mathcal{B}(\pm|\tilde{F}|_{B_\infty(g_i,r) }|),$$
and we are left to prove that for all $1\leq i \leq q$
\begin{equation}\label{Equation_Lift_Bound}
\mathcal{N}_{0,\delta}(\pm|\tilde{F}|_{B_\infty(g_i,r) }|)\leq \frac{(2k-1)^n+1}{2}.
\end{equation}
We will do this by approximating $|\tilde{F}|$ by an element of $\mathcal{S}_{k-1}(\G).$ Namely, (\ref{Equation_threshold}) implies that
$$ C_1 (C_2r)^{k-\frac{Q}{p}} \|F \|_{W^{k,p}_X}< \frac{\delta}{2}.$$
Since $\pi$ is an embedding of $B_{cc}(g_i,  C_2r)$, we have that for all $i$
$$\|\tilde{F}|_{B_{cc}(g_i,C_2r) } \|_{W^{k,p}_X} = \| F|_{\pi(B_{cc}(g_i,C_2r) )} \|_{W^{k,p}_X}\leq  \|F \|_{W^{k,p}_X},$$
and thus
$$C_1 (C_2r)^{k-\frac{Q}{p}} \| \tilde{F}|_{B_{cc}(g_i,C_2r) } \|_{W^{k,p}_X} < \frac{\delta}{2}.$$
Corollary \ref{Corollary_Polynomial_Approximation} implies that for every $1\leq i \leq q$ there exists $S_i\in \mathcal{S}_{k-1}(\G)$ such that
$$d_{C^0}(-|\tilde{F}|_{B_{cc}(g_i,C_2r)}|,-S_i|_{B_{cc}(g_i,C_2r)})=d_{C^0}(|\tilde{F}|_{B_{cc}(g_i,C_2r)}|,S_i|_{B_{cc}(g_i,C_2r)})<\frac{\delta}{2}.$$
Since $B_\infty(g_i,r)\subset B_{cc}(g_i,C_2 r)$, we have that 
$$d_{C^0}(\pm|\tilde{F}|_{B_{\infty}(g_i,r)}|,\pm S_i|_{B_{\infty}(g_i,r)})<\frac{\delta}{2}.$$
Finally, since all $\pm S_i \in \mathcal{S}_{k-1}(\G)$, Proposition \ref{Proposition_BarCount_Stability} and Corollary \ref{Corollary_Barcode_Bound} imply that
$$\mathcal{N}_{0,\delta}(\pm|\tilde{F}|_{B_{\infty}(g_i,r)}|)\leq \mathcal{N}_{0,0}(\pm S_i|_{B_{\infty}(g_i,r)}) \leq \frac{(2k-1)^n+1}{2}.$$
This proves (\ref{Equation_Lift_Bound}) and finishes the proof of the theorem for $d=0.$ To treat the case $d=n-1$, it is enough to notice that by Proposition \ref{Proposition_Duality_Barcodes}
$$\mathcal{N}_{n-1,\delta}(\pm|F|)=\mathcal{N}^{fin}_{n-1,\delta}(\pm|F|) +b_{n-1}(N) = \mathcal{N}^{fin}_{0,\delta}(\mp|F|) +b_{n-1}(N),$$
while by the already-proven inequality in the case $d=0$
$$\mathcal{N}^{fin}_{0,\delta}(\mp|F|) < \mathcal{N}_{0,\delta}(\mp|F|) \leq C \left( \frac{  \|F \|_{W^{k,p}_X}}{\delta} \right)^{\frac{Qp}{kp-Q}}.$$
\end{proof}

\section{Proof of Theorems \ref{Thm:Courant_Nilmanifolds}, \ref{Theorem_Bezout_Main} and Corollary \ref{Corollary_Maximally_Hypoelliptic}}

We will prove both theorems simultaneously.

\begin{proof}
Let $\G$ be a Carnot group, $\Gamma < \G$ a lattice such that $N=\Gamma \backslash \G$, $\pi:G\to N$ the canonical projection and $\mu$ a Haar measure on $G$ which induces $\nu.$ The measure $\mu$ is unique, its density being given locally by the pull-back by $\pi$ of the density of $\nu$. Let $\frak{g}=V_1 \oplus \ldots \oplus V_s$ be a stratification of the Lie algebra of $\G$, $\tilde{X}_1, \ldots , \tilde{X}_m$ a basis of $V_1$ and $X_1=\pi_*(\tilde{X}_1),  \ldots , X_m=\pi_*(\tilde{X}_m).$ Recall that $\varepsilon_k=\frac{Q}{2k-Q}.$ 

Theorem \ref{Thm:Sobolev_Nilmanifolds} with $p=2$ together with (\ref{Inequality_Zeros_Bar}) and (\ref{Inequality_Nodal_Bar}) with $\varepsilon=\frac{\delta}{2}$ implies that for $d\in \{0,n-1 \}$ and all integers $k>\frac{Q}{2}$
\begin{equation}\label{Equation:Bezout_Proof}
z_{d,\delta}(F)\leq \mathcal{N}_{d,\delta} (|F|) \leq C' \left( \frac{ \|F \|_{W^{k,2}_X} }{\delta} \right)^{2\varepsilon_k}+ b_d(N),
\end{equation}
as well as
\begin{equation}\label{Equation:Courant_Prelim}
m_{d,\delta}(F)\leq \mathcal{N}_{d,\frac{\delta}{2}} (-|F|) \leq C'' \left( \frac{ \|F \|_{W^{k,2}_X} }{\delta} \right)^{2\varepsilon_k}+b_d(N),
\end{equation}
where $C''=C' \cdot 2^{2\varepsilon_k}$ and  $C'$ depends only on $\G,\Gamma,\nu,l,k$ and the vector fields $X_1,\ldots,X_m.$ 

As a first step, we wish to strengthen (\ref{Equation:Courant_Prelim}) by removing the $+b_d(N)$ term on the right-hand side of the inequality.  More precisely, we wish to show that 
\begin{equation}\label{Equation:Courant_Improved}
m_{d,\delta}(F)\leq C''' \left( \frac{ \|F \|_{W^{k,2}_X} }{\delta} \right)^{2\varepsilon_k},
\end{equation}
for some $C'''$ which depends only on $\G,\Gamma,\nu,k,l$ and the vector fields $X_1,\ldots,X_m.$ To this end,  let $C_0$ be the constant given by Corollary \ref{Corollary_Basic_Sobolev} with $p=2$ so that
$$\max |F| \leq C_0 \|F \|_{W^{k,2}_X}.$$
If $\delta > C_0  \|F \|_{W^{k,2}_X}$, we have that $\max |F|<\delta$ and thus $\{ |F|>\delta \}=\emptyset.$ It follows that $m_{d,\delta}(F)=0$ and (\ref{Equation:Courant_Improved}) holds for any choice of $C'''>0.$ On the other hand, if $\delta \leq C_0  \|F \|_{W^{k,2}_X}$, we have that $\frac{\|F \|_{W^{k,2}_X}}{\delta} \geq \frac{1}{C_0}$ and the $+b_d(N)$ term can be absorbed into the constant $C'$ which proves (\ref{Equation:Courant_Improved}). 

Corollary \ref{Corollary_Maximally_Hypoelliptic} follows from (\ref{Equation:Bezout_Proof}) and (\ref{Equation:Courant_Improved}) with $k=r$ combined with (\ref{Equation_Maximally_Hypoelliptic_Systems}). To prove Theorems \ref{Thm:Courant_Nilmanifolds} and \ref{Theorem_Bezout_Main}, let us take $\tilde{X}_1,\ldots , \tilde{X}_m$ to be an orthonormal basis of $V_1.$ In this case, the set of vector fields $X=\{X_1,\ldots , X_m\}$ generates the sub-Riemannian structure on $N$, as explained in Example \ref{Example_Gen_Car}.  By (\ref{Equation:Sublaplacians_Equality}), we have that $P=P_{X,\nu}$ and Theorem \ref{Theorem_Subelliptic_Estimates} gives us that
\begin{equation}\label{Estimate_Coordinates}
\|F \|^2_{W^{k,2}_X} = \sum_{i=1}^l \|f_i \|^2_{W^{k,2}_X} \leq \tilde{C}^2 \sum_{i=1}^l \|(1+P)^{\frac{k}{2}}f_i \|^2_{L^2},
\end{equation}
where $\tilde{C}$ depends on $(N,\nu),k$ and the vector fields $X_1,\ldots , X_m.$ Since for all $i$, $f_i \in  \mathcal{F}_\lambda$, we have that $f_i=\sum_{j} a_{i,j} u_j$ where $P u_j = \lambda_j u_j$ with $\lambda_j \leq \lambda$, $\|u_j \|_{L^2}=1$ and all $u_j$ are pairwise orthogonal.  Furthermore, since $\|f_i \|^2_{L^2}=\sum_{j} a_{i,j}^2$, it follows that for all $i$
\begin{equation}\label{Estimate_Orthogonal}
\|(1+P)^{\frac{k}{2}} f_i \|_{L^2}^2 = \Big\Vert \sum_{j} a_{i,j} (1+\lambda_j)^{\frac{k}{2}}u_j \Big\Vert_{L^2}^2 = \sum_{j} a_{i,j}^2 (1+\lambda_j)^{k} \leq  (1+\lambda)^{k}\|f_i \|^2_{L^2}
\end{equation}
Putting the estimates (\ref{Estimate_Coordinates}) and (\ref{Estimate_Orthogonal}) together and using $\|F\|_{L^2}=1$ yields
$$\|F \|_{W^{k,2}_X} \leq \tilde{C} \cdot (1+\lambda)^{\frac{k}{2}}.$$
Combining this estimate with (\ref{Equation:Courant_Improved}) proves Theorem \ref{Thm:Courant_Nilmanifolds}, while combining it with (\ref{Equation:Bezout_Proof}) proves Theorem \ref{Theorem_Bezout_Main}.
\end{proof}

\section{Proof of Theorems \ref{Thm:Sobolev_RotStein} and \ref{Thm:Spectral_RotStein}}

\subsection{Proof of Theorem \ref{Thm:Sobolev_RotStein}}

\begin{proof}
Let us fix an auxiliary Riemannian metric $g$ on $M$ and denote by $\|\cdot \|_{W^{k,p}}$ the usual Sobolev norm with respect to this metric. Theorem 1.12, Remark 1.13 and Remark 1.14 in \cite{BPPPSS22} state that for every $d\geq 0$
\begin{equation}\label{Thm:Sobolev_Original}
\mathcal{N}_{d,\delta} (\pm|F|) \leq \widetilde{C} \left( \frac{\|F \|_{W^{k,p}} }{\delta} \right)^{\frac{n}{k}}+b_d(M),
\end{equation}
with $\widetilde{C}$ depending on $(M,g),l,k,p.$ On the other hand,  by applying Proposition \ref{Proposition_Sobolev_Sobolev} to each $f_i$, we obtain that
\begin{equation}\label{Equation_sk_Sobolev}
\|F \|_{W^{k,p}} \leq C' \|F \|_{W^{sk,p}_X},
\end{equation}
with $C'$ depending only on $(M,g),k,p$, and the vector fields $X_1,\ldots , X_m.$ Putting the two inequalities together yields
\begin{equation}\label{Equation_d_Sobolev}
\mathcal{N}_{d,\delta} (\pm|F|) \leq C \left( \frac{\|F \|_{W^{sk,p}_X} }{\delta} \right)^{\frac{n}{k}}+b_d(M),
\end{equation}
where $C$ depends only on $(M,g),l,k,p$ and the vector fields $X_1, \ldots , X_m.$ Note that in (\ref{Equation_d_Sobolev}), the norm $\|F \|_{W^{sk,p}_X}$ is taken with respect to the measure induced by $g.$ To finish the proof, note that since $M$ is compact, this norm is equivalent to the one defined with respect to $\mu.$ 
\end{proof}

\subsection{Proof of Theorem \ref{Thm:Spectral_RotStein}}

\begin{proof} Let $X_1,\ldots , X_m$ be the generating vector fields for $\xi$ given by Theorem \ref{Theorem_Generating_VF}. We have that $P_\mu = P_{X,\mu}$ by (\ref{Equation:Sublaplacians_Equality}).  As in the proof of Theorem \ref{Thm:Sobolev_RotStein}, let us fix an auxiliary Riemannian metric $g_{R}$ on $M$ and denote by $\|\cdot \|_{W^{k,p}}$ the usual Sobolev norm with respect to this metric.  Theorem 1.1 in \cite{BPPPSS22} states that
$$
m_{d,\delta} (F) \leq \widetilde{C} \left( \frac{\|F \|_{W^{k,p}} }{\delta} \right)^{\frac{n}{k}},
$$
where $\widetilde{C}$ depends on $(M,g_{R}),l,k,p.$ Combining this inequality with (\ref{Equation_sk_Sobolev}) yields
$$
m_{d,\delta} (F) \leq C'' \left( \frac{\|F \|_{W^{sk,p}_X} }{\delta} \right)^{\frac{n}{k}},$$
with constant $C''$ depending on $(M,g_{R}),l,k,p$, and the vector fields $X_1,\ldots , X_m.$ To finish the proof, we argue in the same way as in the proof of Theorems \ref{Thm:Courant_Nilmanifolds} and \ref{Theorem_Bezout_Main}. Namely, we take $p=2$ and the same derivation as (\ref{Estimate_Coordinates}) gives us that
\begin{equation}\label{Equation:Mdelta_General}
m_{d,\delta} (F) \leq \frac{C}{\delta^{\frac{n}{k}}} \left( \sum_{i=1}^l \|(1+P)^{\frac{sk}{2}} f_i \|^2_{L^2}\right)^{\frac{n}{2k}},
\end{equation}
where $C$ depends on $(M,g_{R}),\mu,l,k$ and the vector fields $X_1,\ldots , X_m.$ Proceeding as in (\ref{Estimate_Orthogonal}), we obtain that for all $1\leq i\leq l$
$$\|(1+P)^{\frac{sk}{2}} f_i \|_{L^2}^2 \leq (1+\lambda)^{sk}\|f_i\|^2_{L^2}.$$
Combining this estimate with (\ref{Equation:Mdelta_General}) and using $\|F\|_{L^2}=1$ finishes the proof. 
\end{proof}


\begin{thebibliography}{abbrv}

\bibitem{ABB19} A. Agrachev, D. Barilari and U. Boscain, \emph{A comprehensive introduction to sub-Riemannian geometry}. Cambridge University Press, 2019.

\bibitem{Arnold73} V. Arnold, \emph{The topology of real algebraic curves (works of I.G. Petrovsky and their development)}. Translation of Usp. Mat. Nauk 28:5, 260-262, 1973, translated by Oleg Viro.

\bibitem{BL15} U. Bauer and M. Lesnick, \emph{Induced matchings and the algebraic stability of persistence barcodes}. J. Comput. Geom. 6 (2015), no. 2, 162-191.

\bibitem{BMMS24} U. Bauer, A. M. Medina-Mardones, M. Schmahl, \emph{Persistent homology for functionals}. Commun. Contemp. Math. 26 (2024), no. 10, Paper No. 2350055.

\bibitem{SubRiemannian96} A. Bella\"{i}che and J.-J. Risler (eds.), \emph{Sub-Riemannian geometry}. Progress in Mathematics, vol. 144. Birkh\"{a}user Verlag, Basel (1996).

\bibitem{BCH22} P. B\'{e}rard, P. Charron and B. Helffer, \emph{Non-boundedness of the number of super level domains of eigenfunctions}. J. Anal. Math. 146 (2022), no. 1, 127-164.

\bibitem{BH18} P. B\'{e}rard and B. Helffer, \emph{On Courant's nodal domain property for linear combinations of eigenfunctions. Part I}. Doc. Math. 23 (2018), 1561-1585.

\bibitem{BH20a} P. B\'{e}rard and B. Helffer, \emph{Sturm's theorem on zeros of linear combinations of eigenfunctions}. Expo. Math. 38 (2020), no. 1, 27-50.

\bibitem{BH20b} P. B\'{e}rard and B. Helffer, \emph{Sturm's theorem on the zeros of sums of eigenfunctions: Gelfand's strategy implemented}. Mosc. Math. J. 20 (2020), no. 1, 1-25.

\bibitem{BH21} P. B\'{e}rard and B. Helffer, \emph{On Courant's nodal domain property for linear combinations of eigenfunctions part II}. Schr\"{o}dinger operators, spectral analysis and number theory, 47-88. Springer Proc. Math. Stat. 348 Springer, Cham, 2021.

\bibitem{BH21b} P. B\'{e}rard and B. Helffer, \emph{Level sets of certain Neumann eigenfunctions under deformation of Lipschitz domains application to the extended Courant property}. Ann. Fac. Sci. Toulouse Math. (6) 30 (2021), no. 3, 429-462.


\bibitem{BLU_Book} A.  Bonfiglioli, E. Lanconelli and F. Uguzzoni, \emph{Stratified Lie groups and potential theory for their sub-Laplacians}. Springer Monographs in Mathematics. Springer, Berlin, 2007.



\bibitem{BLS20} L. Buhovsky, A. Logunov and M. Sodin, \emph{Eigenfunctions with infinitely many isolated critical points}. Int. Math. Res. Not. IMRN 2020, no. 24, 10100-10113.

\bibitem{BPPPSS22} L. Buhovsky, J. Payette, I. Polterovich, L. Polterovich, E. Shelukhin and V. Stojisavljevi\'c, \emph{Coarse nodal count and topological persistence}. J. Eur. Math. Soc. (JEMS) 28 (2026), no. 7, 3131-3202.

\bibitem{BPPSS23} L. Buhovsky, I. Polterovich, L. Polterovich, E. Shelukhin and V. Stojisavljevi\'c, \emph{Persistent transcendental B\'{e}zout theorems}. Forum of Mathematics, Sigma, Volume 12, 2024, e72.

\bibitem{CH26} M. Chatzakou and B. Helffer, \emph{Generic simplicity for self-adjoint operators under bounded potential perturbations}. arXiv:2605.31368.

\bibitem{CCBdS16} F. Chazal, W. Crawley-Boevey, V. de Silva, \emph{The observable structure of persistence modules}. Homology Homotopy Appl. 18 (2016), no. 2, 247-265.

\bibitem{CdSGO16} F. Chazal, V. de Silva, M. Glisse, S. Oudot, \emph{The structure and stability of persistence modules}. Springer Briefs Math. Springer, 2016.

\bibitem{CEH07} D. Cohen-Steiner, H. Edelsbrunner, J. Harer, \emph{Stability of persistence diagrams}. Discrete Comput. Geom. 37 (2007), no. 1, 103-120.

\bibitem{CSEHM} D. Cohen-Steiner, H. Edelsbrunner, J. Harer and Y. Mileyko, \emph{Lipschitz functions have {$L_p$}-stable persistence}, Found. Comput. Math. 10 (2010), no. 2, 127-139.

\bibitem{CdVHT22} Y. Colin de Verdi\`{e}re, L. Hillairet and E. Tr\'{e}lat, \emph{Spectral asymptotics for sub-Riemannian Laplacians}, arXiv:2212.02920.

\bibitem{WCB15} W. Crawley-Boevey, \emph{Decomposition of pointwise finite-dimensional persistence modules}. J. Algebra Appl. 14 (2015), no. 5, 1550066.



\bibitem{ELZ02} H. Edelsbrunner, D. Letscher, A. Zomorodian, \emph{Topological persistence and simplification}. Discrete Comput. Geom. 28 (2002), no. 4, 511-533.

\bibitem{EL23} S. Eswarathasan and C. Letrouit, \emph{Nodal sets of eigenfunctions of sub-Laplacians}. Int. Math. Res. Not. IMRN 2023, no. 23, 20670-20700.


\bibitem{FS82} G. B. Folland and E. M. Stein, \emph{Hardy spaces on homogeneous groups}. Mathematical Notes, 28. Princeton University Press, Princeton, NJ; University of Tokyo Press, Tokyo, 1982.


\bibitem{FH25a} R. L. Frank and B. Helffer, \emph{On Courant and Pleijel theorems for sub-Riemannian Laplacians}. Pseudo-differential operators and related topics, 9-23. Trends Math. Res. Perspect. Ghent Anal. PDE Cent. 8 Birkh\"{a}user/Springer, Cham, 2025.

\bibitem{FH25b} R. L. Frank and B. Helffer, \emph{On Courant and Pleijel theorems for sub-Riemannian Laplacians}. J. \'{E}c. polytech. Math. 12 (2025), 1083–1160.

\bibitem{FH26} R. L. Frank and B. Helffer, \emph{Pleijel's theorem for a class of degenerate elliptic operators}. arXiv:2606.04951.

\bibitem{Ginot24} G. Ginot, \emph{Un aperçu des modules de persistance et de leurs applications}. S\'{e}minaire Bourbaki, Vol. 2023/2024, expos\'{e}s 1211-1226, Ast\'{e}risque No. 454 (2024), Exp. No. 1225, 607–641.

\bibitem{GZ03} G. M. L. Gladwell and H.  Zhu, \emph{The Courant-Herrmann conjecture}. Z. Angew. Math. Mech. 83 (2003), no. 4, 275-281.

\bibitem{GSKL26} J. G\'{o}mez-Serrano, R. Koirala and A. Logunov, \emph{Nested nodal loops of biharmonic functions}. arXiv:2605.18699.

\bibitem{Hormander67} L. H\"{o}rmander, \emph{Hypoelliptic second order differential equations}. Acta Math. 119 (1967), 147-171.

\bibitem{Koirala26} R. Koirala, \emph{Nested nodal loops for sums of Laplace eigenfunctions}. arXiv:2605.18705.

\bibitem{Kronrod50} A. S. Kronrod, \emph{On functions of two variables}, (in Russian), Uspekhi Matem. Nauk (N.S.) 35 (1950), 24-134.

\bibitem{LL22} C.  Laurent and M. L\'{e}autaud, \emph{Tunneling estimates and approximate controllability for hypoelliptic equations}. Mem. Amer. Math. Soc. 276 (2022), no. 1357.

\bibitem{LeDonne25} E. Le Donne, \emph{Metric Lie groups—Carnot-Carath\'{e}odory spaces from the homogeneous viewpoint}. Grad. Texts in Math. 306, Springer, Cham, 2025.

\bibitem{LMP23} M. Levitin, D. Mangoubi, I. Polterovich, \emph{Topics in spectral geometry}. Grad. Stud. Math. 237, American Mathematical Society, Providence, RI, 2023.

\bibitem{LW00} G.  Lu and R. L. Wheeden, \emph{High order representation formulas and embedding theorems on stratified groups and generalizations}. Studia Math. 142 (2000), no. 2, 101-133.

\bibitem{LW04} G.  Lu and R. L. Wheeden, \emph{Simultaneous representation and approximation formulas and high-order Sobolev embedding theorems on stratified groups}. Constr. Approx. 20 (2004), no. 4, 647-668.

\bibitem{Metivier76} G. M\'{e}tivier, \emph{Fonction spectrale et valeurs propres d'une classe d'op\'{e}rateurs non elliptiques}. Comm. Partial Differential Equations 1 (1976), no. 5, 467-519.

\bibitem{Mont02} R. Montgomery, \emph{A Tour of Subriemannian Geometries, Their Geodesics and Applications}. Mathematical Surveys and Monographs 91, American Mathematical Society, Providence, RI, 2002.


\bibitem{Oudot15} S. Oudot, \emph{Persistence theory: from quiver representations to data analysis}. Math. Surveys Monogr. 209 American Mathematical Society, Providence, RI, 2015.

\bibitem{Perez22} D. Perez, \emph{On $C^0$-persistent homology and trees}. arXiv:2012.02634.


\bibitem{PPS19} I. Polterovich, L. Polterovich and V. Stojisavljevi\'c, \emph{Persistence barcodes and Laplace eigenfunctions on surfaces}. Geom. Dedicata 201 (2019), 111-138.

\bibitem{PRSZ20} L. Polterovich, D. Rosen, K. Samvelyan, J. Zhang, \emph{Topological persistence in geometry and analysis}. Univ. Lecture Ser. 74 American Mathematical Society, Providence, RI, 2020.

\bibitem{PolSod07} L. Polterovich and M. Sodin, \emph{Nodal inequalities on surfaces}, Math. Proc. Cambridge Philos. Soc. 143 (2007), no. 2, 459-467.

\bibitem{Rag76} M. S. Raghunathan, \emph{Discrete subgroups of Lie groups}. Springer-Verlag, New York-Heidelberg, 1972.

\bibitem{Robinson91} D. W. Robinson, \emph{Elliptic Operators and Lie Groups}. Oxford Math. Monogr. Oxford Sci. Publ. The Clarendon Press, Oxford University Press, New York, 1991.

\bibitem{RS76} L. P. Rothschild and E. M. Stein, \emph{Hypoelliptic differential operators and nilpotent groups}. Acta Math. 137 (1976), no. 3-4, 247-320.

\bibitem{Schmahl22} M. Schmahl, \emph{Structure of semi-continuous $q$-tame persistence modules}. Homology Homotopy Appl. 24 (2022), no. 1, 117-128.

\bibitem{Schmahl25} M. Schmahl, \emph{Comparison of persistent singular and \v{C}ech homology for locally connected filtrations}. Proc. Amer. Math. Soc. 153 (2025), no. 1, 421-436.

\bibitem{Irene26} I. Silvestre-Rosell\'{o}, \emph{Weak and Coarse Courant's theorems for perturbed sub-Laplacians}. PhD thesis, Universit\'{e} de Montr\'{e}al, 2026.

\bibitem{Stoji20} V. Stojisavljevi\'c, \emph{Persistence modules in geometry and dynamics}. PhD thesis, Tel Aviv University, 2020.

\bibitem{Stoji24} V. Stojisavljevi\'c, \emph{Harmonic functions with highly intersecting zero sets}. To appear in Proceedings of the American Mathematical Society.

\bibitem{Street14} B. Street, \emph{Multi-parameter singular integrals}. Ann. of Math. Stud. 189, Princeton University Press, Princeton, NJ, 2014.

\bibitem{VSCC92}  N. Th. Varopoulos, L. Saloff-Coste, T. Coulhon, \emph{Analysis and geometry on groups}. Cambridge Tracts in Math. 100, Cambridge University Press, Cambridge, 1992.

\bibitem{Viro79} O. Y. Viro, \emph{Construction of multicomponent real algebraic surfaces}. Dokl. Akad. Nauk SSSR 248 (1979), no. 2, 279-282.

\bibitem{Vitushkin55} A.G. Vitushkin, \emph{On higher-dimensional variations}. Moscow, 1955.

\bibitem{Yomdin85} Y. Yomdin, \emph{Global bounds for the Betti numbers of regular fibers of differentiable mappings},  Topology 24 (1985), 145-152.

\bibitem{ZC05} A. Zomorodian and G. Carlsson, \emph{Computing persistent homology}. Discrete Comput. Geom. 33 (2005), no. 2, 249-274.


\end{thebibliography}
\end{document}